\documentclass[11pt,reqno]{article} \usepackage{amsmath,authblk}

\usepackage{graphicx,textcomp,subfigure}
\usepackage[top=3cm, bottom=3cm, left=3.5cm, right=3.5cm] {geometry}
\usepackage{amscd,amstext,amsfonts,amsbsy,
amssymb,amsthm,mathrsfs,float,latexsym}
\usepackage{multicol,graphicx,array,multirow,color,lineno}
\usepackage{sidecap,url,fancybox,fancyhdr}
\usepackage[colorlinks=true,linkcolor=blue,citecolor=magenta]{hyperref}
\usepackage{float}

\DeclareMathOperator*{\esssup}{ess\,sup} 
\DeclareMathOperator*{\essinf}{ess\,inf}

\newcommand{\R}{\mathbb R}
\renewcommand{\epsilon}{\varepsilon}
\newcommand{\eps}{\epsilon}

\newcommand{\bra}[1]{\left(#1\right)}
\newcommand{\sbra}[1]{\left[#1\right]}
\newcommand{\na}{\nabla}
\newcommand{\LO}[1]{L^{#1}(\Omega)}
\newcommand{\LQ}[1]{L^{#1}(\Omega_T )}
\newcommand{\dint}[1]{\iint_{\Omega_T}}
\newcommand{\dinttau}[1]{\iint_{\Omega_\tau}}
\newcommand{\dintt}[1]{\iint_{\Omega_t}}
\newcommand{\pa}{\partial}

\newcommand{\Rhat}{\widehat{R}}
\newcommand{\Shat}{\widehat{S}}

\newtheorem{theorem}{{\bf Theorem}}[section]
\newtheorem{definition}[theorem]{Definition}

\theoremstyle{plain} 
\newtheorem{lemma}{Lemma} 
\newtheorem{proposition}{Proposition}[section]

\newtheorem{assumption}{Assumption}[section]
\newtheorem{remark}{Remark}[section]

\usepackage{chemfig,chemmacros}
\usepackage[new]{old-arrows}

\title{\bf \Large  
Singular limit and convergence rate via projection method in a model for plant-growth dynamics with autotoxicity
}

\author{Jeff Morgan$^{a}$\footnote{jmorgan@math.uh.edu, jjmorgan@uh.edu} , Cinzia Soresina$^{b}$\footnote{ cinzia.soresina@unitn.it} , Bao Quoc Tang$^{c}$\footnote{quoc.tang@uni-graz.at, baotangquoc@gmail.com} ,  Bao-Ngoc Tran$^{c,d}$\footnote{bao-ngoc.tran@uni-graz.at}}

\date{}

\begin{document}

\maketitle

\vspace*{-1cm}

\begin{center}
{\small
$^{a}$Department of Mathematics,
	University of Houston, Houston, Texas 77204, USA \vspace{0.2cm}\\ 
$^{b}$Department of Mathematics, University of Trento, Via Sommarive,  14, 38123 Povo TN, Italy \vspace{0.2cm}\\ 
$^{c}$Department of Mathematics and Scientific Computing, University of Graz,  Heinrichstrasse 36, 8010 Graz, Austria 
\\
$^{d}$Department of Mathematics, Faculty of Science, Nong Lam University,  Ho Chi Minh City, Vietnam
}
\end{center}

\begin{abstract}
We investigate a fast-reaction--diffusion system modelling the effect of autotoxicity on plant-growth dynamics, in which the fast-reaction terms are based on the dichotomy between healthy and exposed roots depending on the toxicity. The model was proposed in [Giannino, Iuorio, Soresina, forthcoming] to account for stable stationary spacial patterns considering only biomass and toxicity, and its fast-reaction (cross-diffusion) limit was formally derived and numerically investigated.
In this paper, the cross-diffusion limiting system is rigorously obtained as the fast-reaction limit of the reaction-diffusion system with fast-reaction terms by performing a bootstrap argument involving energies. Then, a thorough well-posedness analysis of the cross-diffusion system is presented, including a $L^\infty_{x,t}$-bound, uniqueness, stability, and regularity of weak solutions. This analysis, in turn, becomes crucial to establish the convergence rate for the fast reaction limit, thanks to the key idea of using an inverse Neumann Laplacian operator. Finally, numerical experiments illustrate the analytical findings on the convergence rate.  

\vspace*{0.5cm}

\end{abstract}

%\tableofcontents  
  
\section{Introduction}
Singular limits, also known as \textit{fast-reaction limits}, early discussed by Evans \cite{evans1980convergence} and Martin \cite{martin1980mathematical}, are essential in reducing a system of differential equations with fast-reaction terms (i.e., multiple time scales) to another with fewer equations and fewer time-scales, often called \emph{limiting system}. In this context, several studies and applications focused on reaction-diffusion systems.

\medskip
 
Interestingly, the mathematical structure of the limiting system can be different from the fast-reaction system. For instance, Stefan moving boundary problems have been derived by Bothe--Pierre \cite{bothe2012instantaneous} and Murakawa--Ninomiya  \cite{murakawa2011fast}, Stefan moving interface problems by Iida-Monobe--Murakawa--Ninomiya \cite{iida2017vanishing}, systems with solutions expressing in terms of Young measures by Perthame--Skrzeczkowski \cite{perthame2022fast}. In Tang--Tran \cite{tang2023rigorous}, the Michaelis--Menten kinetics for the enzymatic reaction is derived rigorously, where the authors showed that not only the number of equations has been reduced but also the kinetics of the system. In other cases, fast-reaction limits lead to cross-diffusion systems, as shown by, for instance, Iida--Mimura--Ninomiya \cite{iida2006diffusion} for the well-known SKT model, Bothe--Pierre--Rolland \cite{bothe2012cross}, Conforto--Desvillettes--Soresina \cite{conforto2018reaction, desvillettes2019non} for predator--prey models, Elia\v{s}--Izuhara--Mimura--Tang \cite{eliavs2022aggregation}, Brocchieri--Corrias--Dietert--Kim \cite{brocchieri2021evolution}, to name a few. Not only number and/or the mathematical structure of the equation changes into the limit, but also the bifurcation structure might be different, as shown by Kuehn--Soresina in \cite{kuehn2020numerical} for SKT model. We refer the reader to Daus--Desvillettes--J\"ungel \cite{daus2020cross, conforto2014rigorous, desvillettes2015new} and Crooks--Du \cite{crooks2022fast} for fast reaction limits with nonlinear diffusion.

\subsection{Problem formulation} 
\label{ProblemFormulation} 
 
In this paper, we focus on the fast-reaction system and its cross-diffusion limit modelling autotoxicity effects in plant-growth dynamics, proposed by Giannino--Iuorio--Soresina \cite{Giannino2024the} to account for stable stationary spatial patterns without water as limiting resource. 

\medskip

Vegetation dynamics, pattern formation and tipping points are crucial 
for assessing environmental health in a climate change scenario. Plant--soil negative feedback gained much attention, see for instance \cite{reynolds1994dynamics,mangan2010negative} and references therein. One important type of feedback is the inducing of autotoxic compounds from decaying plant matter \cite{marasco2014vegetation}, which was not addressed in classical models. By extending the Klausmeier model \cite{klausmeier1999regular,carter2018traveling} with toxicity, pattern formation (or inhomogeneous vegetation structures) \cite{iuorio2021influence} and travelling pulses \cite{carter2024travelling,iuorio2023analysis} were investigated, focusing on the influence of autotoxicity on biomass survival in a specific environment.

\medskip

Recently, in \cite{Giannino2024the}, a new cross-diffusion model for plant-growth dynamics in the presence of autotoxicity effects has been proposed. The Turing instability analysis showed that the cross-diffusion term in the limiting system is the key ingredient to the appearance of spatial patterns, which is not possible with only standard diffusion and without considering water a variable. In particular, the cross-diffusion system is obtained as the fast-reaction limit of a reaction-diffusion systems with \textit{fast-dichotomy}, namely plant (roots) biomass exists in two possible states, namely healthy roots and exposed roots (to toxicity), and it is supposed that the switch between these states occurs on a much faster time scale than their growth/death and decomposition. 

\medskip

This approach has already been exploited in other contexts, such as, for instance, spatial segregation of competing species \cite{iida2006diffusion}, 
predator-prey models \cite{conforto2018reaction,iida2023cross},
Neolithic spread of farmers in Europe \cite{eliavs2021singular}, cockroaches aggregation \cite{eliavs2022aggregation}, and dietary diversity and starvation \cite{brocchieri2021evolution}. This work provides a rigorous analysis of a model of this type.

 \medskip

Let $\Omega$ be a bounded domain in $\mathbb{R}^2$ with sufficiently smooth boundary $\Gamma $, e.g.~$\Gamma$ is of class $C^{2,\alpha}$ for some $\alpha>0$, and $0<T<\infty$. We study the reaction-diffusion model  
\begin{equation}\label{System.OrginalSys}%\tag{S}
	\left\{
	\begin{array}{lll}
	\displaystyle \pa_t R_1^\varepsilon &\hspace*{-0.2cm}=\hspace*{-0.2cm}& \displaystyle d_{R_1} \Delta R_1^\varepsilon + \gamma_1 R_1^\varepsilon(\Rhat-R^\varepsilon)  - \eta_1R_1^\varepsilon - \frac{1}{\eps}\bra{f(S^\varepsilon) R_1^\varepsilon -  g(S^\varepsilon) R_2^\varepsilon  }, \vspace*{0.15cm} \\
		\displaystyle \pa_t R_2^\varepsilon &\hspace*{-0.2cm}=\hspace*{-0.2cm}& \displaystyle  d_{R_2}\Delta R_2^\varepsilon + \gamma_2 R_2^\varepsilon(\Rhat-R^\varepsilon) - \eta_2 R_2^\varepsilon + \frac{1}{\eps} \bra{f(S^\varepsilon) R_1^\varepsilon - g(S^\varepsilon) R_2^\varepsilon }, \vspace*{0.15cm} \\
		\displaystyle \pa_t S^\varepsilon &\hspace*{-0.2cm}=\hspace*{-0.2cm}& \displaystyle d_S\Delta S^\varepsilon + \mu\big(\eta_1R_1^\varepsilon +  \eta_2 R_2^\varepsilon\big) - \rho S^\varepsilon ,   
	\end{array}
	\right.
\end{equation} 
in $\Omega_T:=\Omega \times (0,T)$
%(\bao{Bao: I have a question about modelling: the decay terms $-\eta_1R_1^{\eps}$ and $-\eta_2R_2^{\eps}$ seem to suggest that the roots transform to toxicity (corresponding reactions $R_1 \to S$ and $R_2 \to S$), but in fact they just produce toxicity})
subjected to homogeneous Neumann boundary conditions 
\begin{equation}
		\nabla R_1^\varepsilon \cdot \nu = \na R_2^\varepsilon \cdot \nu = \na S^\varepsilon \cdot \nu = 0 \quad \text{on } \Gamma_T:=\Gamma \times(0,T),
\end{equation}
and nonnegative bounded initial conditions
\begin{equation}
		(R_1^\varepsilon , R_2^\varepsilon , S^\varepsilon )|_{t=0} = (R_{10}  , R_{20} , S_0 ) \quad \text{on } \Omega, 
\label{Condition.Initial}
\end{equation}
where $R_1^\eps=R_1^\eps(x,t)$ is the biomass density of healthy roots, $R_2^\eps=R_2^\varepsilon(x,t)$ the biomass density of roots exposed to the toxicity, and $S^\eps=S^\varepsilon(x,t)$ the toxicity concentration at $(x,t)\in \Omega_T$. 
It is assumed that both roots and toxicity diffuse into the soil. The constants $d_{R_1}$, $d_{R_2}$, $d_S$ are positive diffusion rates of the corresponding species, and it is assumed that the roots exposed to the toxicity diffuse slower than healthy roots, i.e.~$d_{R_2}<d_{R_1}$. The growth of both type of roots is modeled with a logistic growth rate depending on the total biomass $R^\varepsilon:=R^\eps_1+R^\eps_2$ and a carrying capacity $\Rhat$, with coefficients $\gamma_1$ and $\gamma_2$ with $\gamma_2<\gamma_1$. It is also assumed that both types of roots have linear mortality rates with coefficients $\eta_1$ and $\eta_2$ with $\eta_1<\eta_2$. The death and decomposing biomass are then converted into toxicity with conversion factor $\mu$ and toxicity degrades with coefficient $\rho$. The constants $\gamma_1,\, \gamma_2,\, \eta_1,\, \eta_2,\,\mu,\,\rho,\, \Rhat$ are non-negative.
The remaining terms that appear in the system \eqref{System.OrginalSys} are the switch between the two states $R_1^\eps$ and $R_2^\eps$ happening at a faster time-scale. Therefore, the time-scale parameter $\eps$ is supposed to be small, namely $0< \eps \ll 1$. In particular, $f(S^\eps)$ and $g(S^\eps)$ are the non-negative toxicity-dependent switching rates from healthy roots to exposed ones and vice versa. 
It is natural to impose that $f$ is increasing while $g$ is decreasing from $[0,\infty)$ to $[0,\infty)$. 

\medskip

Some typical expressions of $f$ and $g$ are
\begin{itemize}
\item Power functions: $\displaystyle (f(s),g(s))=(s^p,(1+s)^q)$ for $q<0<p$;
\item Holing type II functions: $(f(s),g(s))=\left(\frac{a_1s+b_1}{c_1s+d_1},\frac{a_2s+b_2}{c_2s+d_2}\right)$ for $a_i,b_i,c_i,d_i \ge 0$, $i=1,2$, such that $a_2d_2-b_2c_2<0<a_1d_1-b_1c_1$;  
\item Saturation transition: $\displaystyle (f(s),g(s))=(s,(\Shat -s)_+)$ for a positive constant $\Shat$. 
\end{itemize}

\subsection{Fast reaction limit} Reasonably, as $\varepsilon$ is very small, the system is expected to be close to the critical manifold over the evolution process, where all species tend to attain their steady state driven by their diffusion. This motivates us to investigate the limit of the system as $\varepsilon\to 0^+$. Formally, we expect that  
\begin{equation}
(R_1^\eps,R_2^\eps,S^\eps) \to (R_1,R_2,S) \quad \text{and} \quad f(S^\varepsilon) R_1^\varepsilon -  g(S^\varepsilon)R_2^\varepsilon \to 0,  
\label{Introduction.ConvergencetoCritMani}
\end{equation}
which also yields
\begin{align}
 f(S) R_1 -  g(S)R_2 = 0 . \label{Formula.CriticalManifold}
\end{align}
Denoting by $R$ the limit of $R^\varepsilon:= R_1^\varepsilon + R_2^\varepsilon$, i.e. $R=R_1+R_2$, we get from \eqref{Formula.CriticalManifold} that 
\begin{equation}
\begin{aligned}
	R_1 = \frac{g(S)}{f(S)+g(S)}R =: \xi_1(R,S), \\ R_2 = \frac{f(S)}{f(S)+g(S)}R =: \xi_2(R,S).
	\end{aligned}  \label{Formula.R1R2}
\end{equation}
By adding the first two equations of \eqref{System.OrginalSys}, and sending $\varepsilon\to 0^+$, we formally obtain the following cross-diffusion equation for $R$ and $S$
\begin{equation}\label{System.LimitingSys}%\tag{Limit-S}
	\left\{
	\begin{aligned}
		\pa_t R &= \Delta\sbra{\bra{d_{R_1} - (d_{R_1}-d_{R_2})\frac{f(S)}{f(S)+g(S)}}R} + h(R,S),  \\
		\pa_t S &= d_S\Delta S + \mu \big(\eta_1\xi_1(R,S) +  \eta_2 \xi_2(R,S)\big) - \rho S,  
	\end{aligned}
	\right.
\end{equation}
with
\begin{align}
	\label{Formula.hRS}
	h(R,S):= \sum_{i=1}^2 \left(  \gamma_i \xi_i(R,S) (\Rhat-R)  -  \eta_i\xi_i(R,S) \right) , 
\end{align}
subjected to homogeneous Neumann boundary condition 
\begin{equation}
	\nabla R\cdot \nu = \na S\cdot \nu = 0, \quad x\in\pa\Omega \times (0,T), 
	\label{Condition.LimBoundaryCond}
\end{equation}
and initial data
\begin{equation}
	R(x,0) = R_{10}(x) + R_{20}(x), \quad S(x,0) = S_0(x), \qquad x\in\Omega.	\label{Condition.LimInitialCond}
\end{equation}
The fast reaction system \eqref{System.OrginalSys}-\eqref{Condition.Initial} has been formally reduced to the cross-diffusion system \eqref{Formula.CriticalManifold}-\eqref{Condition.LimInitialCond} by passing $\varepsilon\to  0^+$, where the toxicity $S$ affects the diffusion of the total roots $R$.

In this work, we study the fast reaction limit passing from the fast reaction system \eqref{System.OrginalSys}-\eqref{Condition.Initial} to the limiting system \eqref{Formula.CriticalManifold}-\eqref{Condition.LimInitialCond}, focused on the rigorous derivation of the limiting system, its well-posedness (i.e., global existence, uniqueness, and stability of solutions), and the convergence rate of the limit. Differences in the regularity of solutions to these systems are discussed due to their different structures.

\subsection{Main results and key ideas}

Before presenting our main results, we introduce the following assumption on the initial data and the transition functions, where for a Sobolev space $X$ defined on $\Omega$, $X_+:= \{f\in X: f\ge 0 \text{ a.e. in } \Omega\}$.
\begin{assumption} \label{Assumption.AllinOne} Let $(R_{10},R_{20},S_0)\in L^\infty_+(\Omega)^2 \times W^{2,\infty}_+(\Omega)$, and $f: \R_+ \to \R_+$, $g: \R_+ \to (0,\infty)$ be $C^2$ functions such that $f' \ge 0$, $g'\le 0$ on $\R_+$,  $f'(0)>0$, and 
\begin{align}
\label{Assumption.InitialData}
  \inf_{x\in\Omega } S_0(x) >0. 
\end{align}
%and  
%\begin{align}
%\label{Assumption.AllinOne}  
%M := \sup_{[s_0,\,s_\infty]} \left(   \frac{f'(s)+|f''(s)|}{f(s)^{\kappa}}    +    \frac{|g'(s)|+|g''(s)|}{g(s)^{\kappa}} \right) <\infty , 
%\end{align}
%for some $\kappa<1$. 
\end{assumption}

%At the first glance, the condition \eqref{Assumption.AllinOne} looks quite technical and restricted. However, there is a wide class of couples of transition functions $(f,g)$ fulfilling them. Indeed, the examples of couples $(f,g)$ given in Section \ref{ProblemFormulation} clearly verify this  condition. On the other side, the condition \eqref{Assumption.InitialData} will aid
%to obtain a positively uniform-in-$\varepsilon$ bound of $S^\varepsilon$.  
%{\color{blue} Have a quick comparison with  \cite[System (11)]{desvillettes2015new},  \cite[System (9)]{eliavs2022aggregation}.} 

\medskip

\textbf{Our first main result} rigorously studies the fast reaction limit passing from the reaction-diffusion system  \eqref{System.OrginalSys}-\eqref{Condition.Initial} to the cross-diffusion system \eqref{System.LimitingSys}-\eqref{Condition.LimInitialCond}.  Since this basically requires a priori estimates uniformly in $\varepsilon>0$, energy functions fitting to the structure of the system \eqref{System.OrginalSys}-\eqref{Condition.Initial} are considered. In fact, multiplying the equations for $R_1^\varepsilon$, $R_2^\varepsilon$ by $(f(S^\varepsilon)R_1^\varepsilon)^{p-1}$, $(g(S^\varepsilon)R_2^\varepsilon)^{p-1}$ respectively gives us a reasonable control over the fast reaction term, i.e. the term including $1/\varepsilon$. In other words, we use the energy function
\begin{align}
E_p^\varepsilon(t):=  \int_\Omega \Big( f(S^\varepsilon)^{p-1} (R_1^\varepsilon)^p +  g(S^\varepsilon)^{p-1} (R_2^\varepsilon)^p \Big) dx, \quad t>0,  \label{Function.Energy}
\end{align}
for $1<p<\infty$. Notice that due to the improved duality argument, see \cite{canizo2014improved}, $R_1^\varepsilon, R_2^\varepsilon$ are uniformly bounded in $L^{2+\delta}(\Omega_T)$, for some small constant $\delta>0$. This bound of $R_1^\eps, R_2^\eps$, in combination with the comparison principle and heat regularisation, gives a uniform bound on $\na S^\eps$ in $L^{2(2+\delta)}_{x,t}$.  Our strategy is to use a bootstrap argument as follows: first, we choose $p = 1+\delta$, and look at the evolution of $E^\eps_p$ where the established bounds on $R_1^\eps, R_2^\eps$ and $S^\eps$ are enough to obtain an estimate of the form
\begin{equation*}
	\frac{d}{dt}E_{1+\delta}^{\eps}(t) + \sum_{i=1}^2\int_{\Omega}(R_i^\eps)^{\delta - 1}{|\na R_i^{\eps}|^2}dx \le C_T.
\end{equation*}
This leads to 
\begin{equation*}
	\sum_{i=1}^2\bra{\|R_i^\eps\|_{L^\infty_tL^{1+\delta}_x} + \|\nabla (R_i)^{\frac{\delta+1}{2}}\|_{L^2_{x,t}}} \le C,
\end{equation*}
which consequently implies, by using a two-dimensional interpolation inequality,
\begin{equation*}
	\sum_{i=1}^2{\|R_i^{\eps}\|_{L^{2+2\delta}_{x,t}}} \le C_T.
\end{equation*}
Then, we choose $p = 1 + 2\delta$ and repeat the previous procedure to get
\begin{equation*}
	\sum_{i=1}^2{\|R_i^{\eps}\|_{L^{2+4\delta}_{x,t}}} \le C_T,
\end{equation*}
and therefore, by induction, we get for any $1\le n \in \mathbb N$,
\begin{equation}\label{Introduction.BootstrapArgument}
	\sum_{i=1}^2{\|R_i^{\eps}\|_{L^{2+2^n\delta}_{x,t}}} \le C_{T,n}.
\end{equation}
%
%
%
%Taking into account this point, direct computations give the uniform boundedness of $|\nabla R_i^\varepsilon|^2/(R_i^\varepsilon)^{2-p}$, $i=1,2$, for $p=1+\delta$, which yields that $|\nabla R_i^\varepsilon|$ are uniformly bounded in $L^{1+}(\Omega_T)$, see Desvillettes-Trescases \cite{desvillettes2015new}. However, there is a delightful observation that if 
%\begin{align}
%\label{Assumption.Parameter}
%\frac{\eta_2}{\gamma_2} \mu \Rhat  \le  f(s_0)^{1-\kappa} \bigg( \sup_{[s_0,\,s_\infty]} \frac{f'(s)}{f(s)^{\kappa}} \bigg)^{-1}
%\end{align}
%holds then, cf. Lemma \ref{Lem.Feedback},
%\begin{align*}
%\sum_{i=1}^2 \dint{}  \frac{|\nabla R_i^\varepsilon |^2}{(R_i^\varepsilon)^{1-\frac{3}{2}\delta}}     + \dint{} (R^\varepsilon)^{2+\frac{3}{2}\delta} \le C\bigg( T,\delta, \dint{} (R^\varepsilon)^{2+\delta} \bigg). 
%\end{align*}
%A bootstrap argument can be implemented to see that 
%\begin{align}
%\sum_{i=1}^2 \dint{} \frac{|\nabla R_i^\varepsilon |^2}{(R_i^\varepsilon)^{1-\big(\frac{3}{2}\big)^n\delta}}  + \dint{} (R^\varepsilon)^{2+\big(\frac{3}{2}\big)^n\delta} \le C_{T,\delta,n}, 
%\label{}  
%\end{align}
%which implies 
In this manner, we obtain the uniform boundedness of $\nabla R_1^\varepsilon, \nabla R_2^\varepsilon$ in $L^2(\Omega_T)$, and  $R_1^\varepsilon,R_2^\varepsilon$ in $L^p(\Omega_T)$ for any $1\le p<\infty$. Hence, the $\varepsilon$-uniform regularity of $(R_1^\varepsilon,R_2^\varepsilon,S^\varepsilon)$ is obtained in $L^p(\Omega_T)^2 \times W^{2,1}_p(\Omega_T)$, cf. Lemma \ref{Lem.Bootstrap}. 

\medskip

Our first main result is given in the following theorem.

\begin{theorem}[Fast reaction limit]    
\label{Theo.ConvergeToLimitSys} Suppose Assumption \ref{Assumption.AllinOne}. Let $(R_{1}^\varepsilon,R_{2}^\varepsilon,S^\varepsilon)$ be the classical solution to the fast reaction-diffusion system  \eqref{System.OrginalSys}-\eqref{Condition.Initial} for each $\varepsilon>0$. Then, up to subsequences (not relabeled),
	\begin{equation}
(R_{1}^\varepsilon,R_{2}^\varepsilon,S^\varepsilon) \to  (R_{1},R_{2},S)\quad  \text{in } \LQ{k}^2 \times L^\infty(\Omega_T) , \label{Convergence.Solutions}  
	\end{equation}
as $\eps \to 0^+$ for any $1\le k<\infty$, where $(R,S):=(R_1+R_2,S)$ is a weak solution to the limiting system \eqref{System.LimitingSys}-\eqref{Condition.LimInitialCond} in the sense of Definition \ref{Definition.WeakSolution}. In addition, $R_1$, $R_2$ can be  expressed in terms of $R$ and $S$ as \eqref{Formula.R1R2}. The convergence to the critical manifold is given by
\begin{equation}
\|f(S^\varepsilon) R_1^\varepsilon   - g(S^\varepsilon) R_2^\varepsilon\|_{L^q(\Omega_T)} \le C_T \varepsilon^{1/q}, \quad 2\le q<\infty.
\label{ConvergenceToCritMani}
\end{equation} 
\end{theorem} 

\begin{remark} \label{Remark.FastReactionLim}

%\text{ }

%\begin{enumerate} 

The uniform boundedness of $R_1^\varepsilon,R_2^\varepsilon$ in $L^\infty(\Omega_T)$ is not directly obtained from \eqref{Introduction.BootstrapArgument} because of the implicit dependence of $C_{T,n}$ on $n$. In other words, $\lim_{n\to \infty} C_{T,n}$ may not be finite.  

%\end{enumerate}
\end{remark}

Since the limiting system includes cross-diffusion, a structure different from the fast reaction-diffusion system, a better insight into its well-posedness is needed to reveal the difference between the regularity of $R,S$ and the $\varepsilon$-uniform regularity of $R^\varepsilon,S^\varepsilon$. By observing the logistic structure, we can perform a delicate Alikakos iteration \cite{alikakos1979application}, which results in an $L^\infty(\Omega_T)$-boundedness for $R$, cf. Lemma \ref{Lem.LInfinityEstimate}. This allows us to get the pointwise estimate 
\begin{align*}
\sum_{i=1}^2 |\xi_i(\widetilde R,\widetilde S)-\xi_i(R,S)|   \le C_{T}  \big( |\widetilde R - R| +|\widetilde S-S| \big), 
\end{align*}
where $(\widetilde{R},\widetilde{S})$ is another weak solution to the limiting system, cf. \eqref{Proof.LInfDifferenceStep5}, and $\xi_i$ are defined in \eqref{Formula.R1R2}. Then, we can control the difference $\|\widetilde S - S\|_{L^\infty(\Omega_t)}$ by $\|(\widetilde R,\widetilde S) - (R,S)\|_{L^2(\Omega_t)}$ and $\|\nabla((\widetilde R,\widetilde S) - (R,S))\|_{L^2(\Omega_t)}$ as discussed in Lemma \ref{Lem.LInfDifference}. 
Therefore, the mixed gradient term including $\nabla R, \nabla \widetilde R$ can be controlled, cf. Lemma \ref{Lem.DifficultTerm}. With this point, we can combine $L^p-L^q$ estimates for the heat semigroup and the heat regularisation to prove the uniqueness and stability of the weak solution.    
Since the weak solution is smooth enough, especially the component $S$, it suggests to rewrite the equation for $R$ as  
\begin{align*}
\partial_t R - a(S)\Delta R - b(S)\cdot \nabla R + c(S)R = h(R,S) ,
\end{align*} 
where the coefficients $a(S)$, $b(S)$, $c(S)$ do not depend on $R$. The smoothness of $S$ ensures that $a(S)$ is positive, bounded, and H\"older continuous, and $b(S),c(S)$ are smooth.  Thus, the weak solution is regular and satisfies the limiting system in the strong sense.

\begin{theorem}[Well-posedness of the limiting system]
\label{Theo.WellposedLimitSys}
The weak solution $(R,S)$ to \eqref{System.LimitingSys}-\eqref{Condition.LimInitialCond} obtained by Theorem \ref{Theo.ConvergeToLimitSys} is unique and enjoys the higher regularity%. Then,  for any $1\le q<\infty$
\begin{align}
\|S\|_{W^{2,1}_q(\Omega_T)} +   \|R\|_{L^\infty(\Omega_T)} + \|R\|_{L^2(0,T;H^1(\Omega))} \le C_T, 
\label{Regularity.LimitingSol}
\end{align}
for any $1\le q<\infty$. Furthermore, we have the following stability estimate: If $(\widetilde{R},\widetilde{S})$ is the weak solution to \eqref{System.LimitingSys}-\eqref{Condition.LimInitialCond} with respect to the given initial data $(\widetilde{R}_0,\widetilde{S}_0)$ then
\begin{equation}
\begin{aligned}
\|(\widetilde R,\widetilde S)-(R,S)\|_{L^{2}(0,T;H^1(\Omega)) \times W^{2,1}_{2}(\Omega_T)} \le C_T \|(\widetilde R_0,\widetilde S_0)-(R_0,S_0)\|_{L^{2}(\Omega) \times W^{2,1+}_{2+}(\Omega)} .         
\end{aligned} 
\label{Estimate.Stability}
\end{equation}
 The couple $(R,S)$ is also the strong solution to \eqref{System.LimitingSys}-\eqref{Condition.LimInitialCond}, and satisfies       
\begin{align}
\|R\|_{W^{2,1}_q(\Omega_T)} + \|S\|_{W^{2,1}_q(\Omega_T)}
   \le C_T.  
\label{STRONGregularity}
\end{align}  
\end{theorem}

\textbf{Our second main result} is the convergence rate of the fast reaction limit, namely, the rates $(|R^\varepsilon-R|,|S^\varepsilon-S|)$. One of the biggest challenges here is the lack of an estimate for the gradient convergence  
$\nabla(f(S^\varepsilon) R_1^\varepsilon -  g(S^\varepsilon)R_2^\varepsilon)\to 0$ and its rate. This naturally arises when considering energy functions for the equation  
\begin{align}
\label{Equation.Reps-R}
 \partial_t (R^\varepsilon-R) =   \Delta \left[ \sum_{i=1}^2 d_{R_i}(R_i^\varepsilon - \xi_i(R,S)) \right] +  (h(R^\varepsilon,S^\varepsilon)- h(R,S)).  
\end{align}
To overcome the difficulty, the main idea is to project the above equation onto  $H^2(\Omega)^*$. More precisely, by considering the Neumann Laplacian $-\Delta_\zeta:=-(\Delta-\zeta I)$ with a fixed constant $\zeta>0$, we can imply from \eqref{Equation.Reps-R} that  
\begin{align}
\label{Equation.U}
 \partial_t U^\varepsilon =  (I+\zeta \Delta_\zeta^{-1})\Delta_\zeta X^\varepsilon + Y^\varepsilon,  
\end{align}
where 
$U^\varepsilon$, and $X^\varepsilon$, $Y^\varepsilon$ are respectively the images of $R^\varepsilon - R$,  $\sum_{i=1}^2 d_{R_i}(R_i^\varepsilon - \xi_i(R,S))$, and $h(R^\varepsilon,S^\varepsilon) - h(R,S)$  after acting $\Delta_\zeta^{-1}$. Based on the equation \eqref{Equation.U}, the convergence \eqref{ConvergenceToCritMani} is enough to construct an estimate for the rate $(|R^\varepsilon-R|,|S^\varepsilon-S|)$.          
 
\begin{theorem}[Convergence rate of the fast reaction limit]
\label{Theo.ConRate}  
For each $\varepsilon>0$ let $(R_{1}^\varepsilon,R_{2}^\varepsilon,S^\varepsilon)$ be the classical solution to \eqref{System.OrginalSys}-\eqref{Condition.Initial}, and $(R,S)$ be the strong solution to \eqref{System.LimitingSys}-\eqref{Condition.LimInitialCond}. Then   
\begin{align}\label{convergence_rate}
\|R^\varepsilon-R\|_{L^2(\Omega_T)} + \|S^\varepsilon-S\|_{W^{2,1}_2(\Omega_T)}  \le C_T \sqrt{\varepsilon}. 
\end{align} 
\end{theorem}

\begin{remark} \label{Remark.ConvergenceRate}

\text{ }

\begin{enumerate} 

\medskip

\item In Iida-Mimura-Ninomiya \cite{iida2006diffusion}, a condition result was presented by imposing an assumption that the $\varepsilon$-dependent solution is uniformly bounded in $L^\infty(\Omega_T)$.

\item By interpolating the convergence \eqref{convergence_rate} with \eqref{Convergence.Solutions}, we can get the convergence rate in $\LQ{p}$ for any $1\le p <\infty$.

\end{enumerate}
\end{remark}

\medskip

\noindent\textbf{Organisation of the paper}.
The rest of this paper is organised as follows. In Section \ref{sec2}, the fast reaction limit passing from the fast reaction diffusion system \eqref{System.OrginalSys}-\eqref{Condition.Initial} to the cross-diffusion system \eqref{System.LimitingSys}-\eqref{Condition.LimInitialCond} is studied rigorously. In Section \ref{sec3}, we present a complete insight into the well-posedness of the system \eqref{System.LimitingSys}-\eqref{Condition.LimInitialCond}. An estimate for the convergence rate of the fast reaction limit is obtained in Section \ref{sec4}. Finally, in Section \ref{sec5}, we implement some numerical examples to illustrate the proposed analysis of the convergence rate.    
 
\medskip 

\noindent {\bf Notations.} We use one and the same symbol $C$ without distinction
in order to denote positive constants, which may change line by line. In some expressions, the dependence on $T,\delta,n,\dots$ will be emphasised by writing $C_{T,\delta,n,\dots}$.  

\section{Fast reaction limit}\label{sec2}

\subsection{Well-posedness of the  $\varepsilon$-dependent system}

\begin{proposition}[Well-posedness of  the $\varepsilon$-dependent system]  
	Let $\eps>0$.
 %and $\Omega \subset\mathbb R^2$ be a bounded domain with smooth boundary $\pa\Omega$. 
 Assume that $f, g: \R_+ \to \R_+$ are locally Lipschitz continuous functions. Then, for any nonnegative initial data $R_{10}^\eps, R_{20}^\eps, S_0^\eps\in L^\infty(\Omega)$, there exists a unique global strong solution to system \eqref{System.OrginalSys} in the following sense: for $Z\in \{R_1^\eps, R_2^\eps, S^\eps\}$, $Z\in C([0,\infty);L^p(\Omega))\cap L^{\infty}_{\text{loc}}(\mathbb R_+; L^\infty(\Omega))$, and if $p>2$ and $T>0$ then $Z \in W^{(2,1)}_p(\Omega_T)$ for all $p>2$, $T>0$, and the equations in \eqref{System.OrginalSys} are satisfied for a.e. $(x,t)\in \Omega_T$.
	
	\medskip
	Moreover, since $f, g$ are smooth, the solution is smooth for positive time, i.e. for $Z\in \{R_1^\eps, R_2^\eps, S^\eps\}$, $Z\in C^{2,1}(\Omega\times (\tau,T))\cap C(\overline{\Omega}\times[\tau,T])$ for any $0<\tau<T$, and the equations in \eqref{System.OrginalSys} are satisfied everywhere.  
\end{proposition}
\begin{proof}
	Since the nonlinearities of \eqref{System.OrginalSys} are locally Lipschitz continuous and satisfy the quasi-positivity condition: $f_Z(R_1^\eps, R_2^\eps, S^\eps) \ge 0$ for $\{R_1^\eps, R_2^\eps, S^\eps\} \in \mathbb{R}_+^3$ with $Z=0$, where $f_Z(\cdot)$ is the nonlinearity of the equation for $Z$ in \eqref{System.OrginalSys}, we can apply standard results for reaction-diffusion systems, for instance, \cite{rothe2006global}, to obtain a unique nonnegative local strong solution on the maximal interval $(0,T_{\max})$. Moreover, we have the blow-up criteria
	\begin{equation*}
		T_{\max}<+\infty \quad \Rightarrow \quad \lim_{t\uparrow T_{\max}}\bra{\|R_1^\eps(t)\|_{\LO{\infty}} + \|R_2^\eps(t)\|_{\LO{\infty}} + \|S^\eps(t)\|_{\LO{\infty}}} = \infty.
	\end{equation*}
	We show the global existence by contradiction. Assume $T_{\max}<+\infty$. Let $T\in(0,T_{\max})$. In the following, we will denote by $C(T)$ a constant depending on $T>0$ that is finite for all $T<\infty$. Summing the equations of $R_1^\eps$ and $R_2^\eps$ leads to
	\begin{equation*}
		\pa_t(R_1^\eps + R_2^\eps) = \Delta (d_{R_1}R_1^\eps + d_{R_2}R_2^\eps) + \max\{\gamma_1,\gamma_2\}\widehat{R}(R_1^\eps + R_2^\eps).
	\end{equation*}
	By applying the improved duality lemma, see e.g. \cite{canizo2014improved,einav2020indirect}, we have
	\begin{equation*}
		\|R_1^\eps\|_{\LQ{2+\delta}} + \|R_2^\eps\|_{\LQ{2+\delta}} \le C(T)
	\end{equation*}
	for some $\delta>0$. From the equation of $S^\eps$, we have
	\begin{equation*}
		\pa_t S^\eps - d_S\Delta S^\eps + \rho S^\eps = \mu(\eta_1R_1^\eps + \eta_2R_2^\eps) \in \LQ{2+\delta}.
	\end{equation*}
	Using the smoothing effect of the heat operator, see e.g. \cite{einav2020indirect}, we obtain
	\begin{equation*}
		\|S^\eps\|_{\LQ{\infty}} \le C_T.
	\end{equation*}
	Due to the continuity of $f$ and $g$, 
	\begin{equation*}
		\|f(S^\eps)\|_{\LQ{\infty}} + \|g(S^\eps)\|_{\LQ{\infty}} \le C_T.
	\end{equation*}
	Thus
	\begin{equation*}
		\pa_t R_1^\eps - d_{R_1}\Delta R_1^\eps \le \gamma_1 \widehat{R} R_1^\eps + \frac 1\eps g(S^\eps)R_2^\eps \le C_T (R_1^\eps + R_2^\eps) \in \LQ{2+\delta}.
	\end{equation*}
	By the smoothing effect of the heat operator and comparison principle, we obtain
	\begin{equation*}
		\|R_1^\eps\|_{\LQ{\infty}} \le C_T.
	\end{equation*}
	The same argument applied to the equation of $R_2^\eps$ gives $  \|R_2^\eps\|_{\LQ{\infty}} \le C_T$. Therefore,
	\begin{equation*}
		\sup_{T\in(0,T_{\max})}\bra{\|R_1^\eps(t)\|_{\LO{\infty}} + \|R_2^\eps(t)\|_{\LO{\infty}} + \|S^\eps(t)\|_{\LO{\infty}}} \le C(T_{\max}) <+\infty,
	\end{equation*}
	which is a contradiction. Therefore, $T_{\max} = \infty$. Since  $f$ and $g$ are smooth, we can exploit the smoothing effect of the heat operator to see that for positive time, the solution is in fact infinitely differentiable and satisfies the system \eqref{System.OrginalSys} in the classical sense for any $t>0$, see e.g. \cite{amann1985global}.
\end{proof}

\subsection{Uniform-in-$\varepsilon$ bounds}
 
 This section begins with necessarily uniform estimates for solutions to the $\varepsilon$-dependent system given in Lemma \ref{Lem.UniformRegularity}. For the sake of convenience, we will say \textit{uniform/uniformly} to indicate uniform/uniformly with respect to $\varepsilon>0$. These uniform estimates will support the so-called modified energy estimate in Lemma \ref{Lem.EnergyEstimate}. Finally, we will establish rigorous convergence of the $\varepsilon$-dependent system to the limiting system in Subsection \ref{Sec.FRLsmoothcase}. 

We first need the following technical lemma, which holds for any dimension $N\ge 1$ and any bounded domain $\Omega$ with sufficiently smooth boundary.

\begin{lemma}[{\cite[Corollary of Theorem 9.1, Chapter IV]{ladyzhenskaia1988linear}}] \label{Lem.HeatRegularisation}
	  Given $1<q<\infty$ and $d>0$. Assume $f\in L^q(\Omega_T )$, $u_0\in W^{2-2/q,q}(\Omega)$, and $u$ is the   solution to the problem 
	  \begin{equation}	\nonumber
 \begin{cases}
		\partial_t u - d \Delta u = f, &\text{ in }  \Omega_T,\\
		\nabla u \cdot \nu = 0, &\text{ on } \Gamma_T,\\
		u(x,0) = u_0(x), &\text{ in } \Omega.
	\end{cases}
	\end{equation}
 Then, $u\in W^{(2,1)}_p(\Omega_T)$ and for 
	\begin{align*}
	p=\left\{ \begin{array}{lllllll}
	%%%%%1
	 \dfrac{(N+2)q}{N+2-2q}  & \text{if }  q<\frac{N+2}{2}, \vspace*{0.15cm}\\
	\in[1,\infty) \text{ arbitrary}  & \text{if }  q=\frac{N+2}{2}, \vspace*{0.15cm}\\
	\infty   & \text{if }  q>\frac{N+2}{2}, \vspace*{0.15cm}
	\end{array} 
	\right.
	\quad  
	r=\left\{ \begin{array}{lllllll}
	%%%%%1
	\dfrac{(N+2)q}{N+2-q}  & \text{if }  q<N+2, \vspace*{0.15cm}\\
	\in[1,\infty)\text{ arbitrary}  & \text{if }  q=N+2, \vspace*{0.15cm}\\
	\infty   & \text{if }  q>N+2 , \vspace*{0.15cm}
	\end{array} 
	\right.
	\end{align*}
we have 
\begin{align*}
	&\|u\|_{L^{p}(\Omega_T )} + \|\nabla u\|_{L^{r}(\Omega_T )^2} + \|\partial_t u\|_{L^{q}(\Omega_T )} + \|\Delta u\|_{L^{q}(\Omega_T )}  \\
 & \hspace{5cm} \le C_T \left(  \|f\| _{L^q(Q_{T})} + \| u_0\|_{W^{2-2/q,q}(\Omega)} \right) , \nonumber 
\end{align*} 
where $C_T$ depends only on $d,q,N,\Omega, T$. 
\end{lemma} 
 
\begin{lemma}[Uniform regularity] \label{Lem.UniformRegularity}  There exists $\delta>0$  such that 
\begin{align}
\sup_{\varepsilon>0} \left( \|R_1^\varepsilon\|_{L^{2+\delta }(\Omega_T)} + \|R_2^\varepsilon\|_{L^{2+\delta }(\Omega_T)} \right) < \infty , \label{Bound.UniR1R2}
\end{align}
which implies
\begin{align}
\sup_{\varepsilon>0} \left(  \|\partial_t S^\varepsilon\|_{L^{2+\delta }(\Omega_T)} + \|\Delta S^\varepsilon\|_{L^{2+\delta }(\Omega_T)} + \|S^\varepsilon\|_{L^{\infty}(\Omega_T)} + \|\nabla S^\varepsilon\|_{L^{\frac{8+4\delta }{2-\delta }}(\Omega_T)^2} \right) < \infty . \label{Bound.UniT}
\end{align} 
\end{lemma}

\begin{proof} By adding the equations for $R_1^\varepsilon$, $R_2^\varepsilon$,  we have 
\begin{align*}
\partial_t(R_1^\varepsilon +R_2^\varepsilon) - \Delta (d_{R_1} R_1^\varepsilon + d_{R_2} R_2^\varepsilon) \le \gamma_1 R_1^\varepsilon   + \gamma_2 R_2^\varepsilon.    
\end{align*}
Since $R_{10}, R_{20}\in L^{\infty}(\Omega)$, thanks to improved duality, see e.g. \cite{canizo2014improved,einav2020indirect}, there exists $\delta>0$ such that $R_1^\varepsilon$, $R_2^\varepsilon$ are uniformly bounded in $L^{2+\delta }(\Omega_T)$, i.e.  \eqref{Bound.UniR1R2} holds. Taking this into account and then the comparison principle, we can apply the heat regularisation given by Lemma \ref{Lem.HeatRegularisation} and imply from  
\begin{align*}
\partial_t S^\varepsilon - d_S \Delta S^\varepsilon \le \mu(\eta_1R_1^\varepsilon + \eta_2 R_2^\varepsilon) 
\end{align*} 
that $S^\varepsilon$ is uniformly bounded in $L^{\infty}(\Omega_T)$. Therefore, we have a uniform bound in $L^{2+\delta }(\Omega_T)$ for the term  $\mu\big(\eta_1R_1^\varepsilon +  \eta_2 R_2^\varepsilon\big) - \rho S^\varepsilon$, which directly deduces the uniform estimate   \eqref{Bound.UniT} after applying the heat regularisation    with  $S_0\in W^{2,\infty}(\Omega)$. 
\end{proof}

\begin{lemma}[Lower bound for prey] \label{Lem.LowerBound} Assumption \eqref{Assumption.InitialData} ensures that 
\begin{align}
\inf_{\varepsilon>0} \left( \essinf\limits_{(x,t)\in   \Omega_T } S^\varepsilon(x,t) \right) \ge  s_0>0, \label{Estimate.T0LowerBound}
\end{align}
where $s_0:=e^{-\rho T} \inf_{x\in\Omega } S_0(x)$. 
\end{lemma}

\begin{proof} By the comparison principle, for each $\varepsilon$, $S^\varepsilon$ is bounded from below by the solution  to 
 \begin{equation}	\nonumber
 \begin{cases}
		\partial_t S_b^\varepsilon  - d_S \Delta S_b^\varepsilon  = -\rho S_b^\varepsilon , &\text{ in }  \Omega_T,\\
		\nabla S_b^\varepsilon  \cdot \nu = 0, &\text{ on } \Gamma_T,\\
		S_b^\varepsilon  (x,0) = S_0(x), &\text{ in } \Omega,
	\end{cases}
	\end{equation}
where we recall the non-negativity of $R_1^\varepsilon$, $R_2^\varepsilon$. Therefore,  
$$ 
S^\varepsilon \ge S_b^\varepsilon \ge e^{-\rho T} \inf_{x\in\Omega } S_0(x) >0$$ 
due to  Assumption \eqref{Assumption.InitialData} and basic estimate for linear heat equations.  
\end{proof}

\subsection{Bootstrap via energy estimate}

We will propose a bootstrap argument to improve the uniform regularity of the $\varepsilon$-solution and its gradients, and to get an important estimate for the fast reaction term. The argument is based on studying the energy function \eqref{Function.Energy}, the heat regularisation given by Lemma \ref{Lem.HeatRegularisation},  and interpolation estimates.  

\begin{lemma}[Energy estimate] 
\label{Lem.EnergyEstimate}
For all $1<p<\infty$, it holds that  
\begin{equation}
\begin{aligned}
  \frac{d}{dt}E_p^\varepsilon(t)  &  + C_{d_{R_1},p} \int_\Omega  |\nabla (R_1^\varepsilon)^{\frac{p}{2}} |^2 + C_{d_{R_2},p} \int_\Omega |\nabla (R_2^\varepsilon)^{\frac{p}{2}} |^2  +  \frac{1}{\eps} \int_\Omega  Q_p^\varepsilon \\
&\le C_{p,T} \left( E_p^\varepsilon(t) +  \int_\Omega  (R_1^\varepsilon)^{p} |\nabla S^\varepsilon|^2  +  \int_\Omega (R^\varepsilon)^{p+1} \right) ,  
\end{aligned} \label{Estimate.Energy}
\end{equation}
where 
\begin{align*}
Q_p^\varepsilon:= \Big( (f(S^\varepsilon) R_1^\varepsilon )^{p-1} -
\big( g(S^\varepsilon) R_2^\varepsilon \big)^{p-1} \Big) \Big( f(S^\varepsilon)   R_1^\varepsilon   - g(S^\varepsilon)  R_2^\varepsilon) \Big). 
\end{align*}
Moreover, the constant $C_{p,T}$ depends on $p$ and remains finite for finite values of $p$.
\end{lemma}

\begin{proof}    By using the equations for $R_1^\varepsilon$ and $R_2^\varepsilon$, the following expression is straightforward   
\begin{align*}
\frac{dE_p^\varepsilon}{dt}	
	&=    (p-1) \int_\Omega  \Big( f(S^\varepsilon)^{p-2} (R_1^\varepsilon)^p \partial_t f(S^\varepsilon) +  g(S^\varepsilon)^{p-2} (R_2^\varepsilon)^p \partial_t g(S^\varepsilon) \Big) \\ 
	& +  p \int_\Omega  \Big( d_{R_1}  ( f(S^\varepsilon) R_1^\varepsilon)^{p-1} \Delta R_1^\varepsilon +  d_{R_2} \big( g(S^\varepsilon) R_2^\varepsilon \big)^{p-1} \Delta R_2^\varepsilon   \Big) \\
&+  p\int_\Omega  ( f(S^\varepsilon) R_1^\varepsilon)^{p-1} \Big( \gamma_1 R_1^\varepsilon(\Rhat- R^\varepsilon)  - \eta_1 R_1^\varepsilon \Big)  \\
&+    p \int_\Omega \big(g(S^\varepsilon) R_2^\varepsilon\big)^{p-1} \Big( \gamma_2
R_2^\varepsilon(\Rhat- R^\varepsilon) - \eta_2 R_2^\varepsilon \Big) \\
& -  \frac{p}{\eps  } \int_\Omega  \Big( (f(S^\varepsilon) R_1^\varepsilon )^{p-1} -
\big( g(S^\varepsilon) R_2^\varepsilon \big)^{p-1} \Big) \Big( f(S^\varepsilon)   R_1^\varepsilon   - g(S^\varepsilon)  R_2^\varepsilon) \Big) .
\end{align*} 
Then, integration by parts shows
\begin{equation}
\begin{aligned}
 \frac{1}{p-1}\frac{dE_p^\varepsilon}{dt}  	& + \frac{p}{p-1}  \frac{1}{\eps} \int_\Omega \Big( (f(S^\varepsilon) R_1^\varepsilon )^{p-1} -
\big( g(S^\varepsilon) R_2^\varepsilon \big)^{p-1} \Big) \Big( f(S^\varepsilon)   R_1^\varepsilon   - g(S^\varepsilon)  R_2^\varepsilon) \Big)  \\
& +  
p d_{R_1} \int_\Omega  f(S^\varepsilon)^{p-1} (R_1^\varepsilon)^{p-2} |\nabla R_1^\varepsilon |^2  
 + p d_{R_2} \int_\Omega  g(S^\varepsilon)^{p-1} (R_2^\varepsilon)^{p-2} |\nabla R_2^\varepsilon |^2   \\
& =   \int_\Omega f(S^\varepsilon)^{p-2} (R_1^\varepsilon)^p \partial_t f(S^\varepsilon)  -  p d_{R_1} \int_\Omega f(S^\varepsilon)^{p-2} (R_1^\varepsilon)^{p-1}  \nabla f(S^\varepsilon) \cdot \nabla R_1^\varepsilon    \\ 
 & +   \int_\Omega g(S^\varepsilon)^{p-2} (R_2^\varepsilon)^p \partial_t g(S^\varepsilon)   -  p	  d_{R_2} \int_\Omega  g(S^\varepsilon)^{p-2} (R_2^\varepsilon)^{p-1} \nabla  g(S^\varepsilon) \cdot \nabla R_2^\varepsilon   \\
 & + \frac{p}{p-1} \int_\Omega (f(S^\varepsilon) R_1^\varepsilon )^{p-1} \Big( \gamma_1 R_1^\varepsilon(\Rhat- R^\varepsilon)  - \eta_1 R_1^\varepsilon \Big)  \\
& + \frac{p}{p-1} \int_\Omega  (g(S^\varepsilon) R_2^\varepsilon)^{p-1} \Big( \gamma_2
R_2^\varepsilon(\Rhat- R^\varepsilon) - \eta_2 R_2^\varepsilon \Big). 
\end{aligned} \label{Calculate.EnergyStep1}  
\end{equation}
We will estimate terms on the right-hand side. Due to Lemmas \ref{Lem.UniformRegularity}, \ref{Lem.LowerBound}, 
\begin{align}
0< s_0 \le S^\varepsilon(x,t) \le s_\infty<\infty ,  \label{Estimate.TwoSideBoundofSeps}
\end{align}
where $s_\infty:=\sup_{\varepsilon>0}\|S^\varepsilon\|_{L^\infty(\Omega_T)}$. 
By Assumption \ref{Assumption.AllinOne} we have $f(s_\infty)\ge f(s_0)>0$ and $g(s_0)\ge g(s_\infty)>0$. Hence, the assumption $f,g \in C^2(\R_+)$ guarantees      
\begin{equation}
\begin{aligned}
f(S^\varepsilon)^{k} \ge \sup_{s_0\le s\le s_\infty} f(s)^{k} &\ge  \min\{f(s_0)^{k};f(s_\infty)^{k}\} \ge C_{k,T} >0, \\
g(S^\varepsilon)^{k} \ge \sup_{s_0\le s\le s_\infty} g(s)^{k} &\ge  \min\{g(s_0)^{k};g(s_\infty)^{k}\} \ge C_{k,T} >0, 
\end{aligned}
\label{Calculate.EnergyStep2a}  
\end{equation}
for $k\in \mathbb{R}$, and for the derivatives of $f$ and $g$
\begin{align*}
|f^{(l)}(S^\varepsilon)| \le \sup_{s_0\le s\le s_\infty} |f^{(l)}(s)|  &\le   \sup_{s_0\le s\le s_\infty} |f^{(l)}(s)| \le C_T, \\
|g^{(l)}(S^\varepsilon)| \le \sup_{s_0\le s\le s_\infty} |g^{(l)}(s)|  &\le   \sup_{s_0\le s\le s_\infty} |g^{(l)}(s)| \le C_T,
\end{align*}
for $l\in \{0;1;2\}$. Therefore, we can employ the equation for $S^\varepsilon$ to estimate the terms including $\partial_t f(
S^\varepsilon)$ and $\partial_t g(S^\varepsilon)$. Indeed,   
\begin{equation}
\label{Calculate.EnergyStep2}  
\begin{aligned}
&  \int_\Omega  f(S^\varepsilon)^{p-2} (R_1^\varepsilon)^p \partial_t f(S^\varepsilon) =   \int_\Omega  f(S^\varepsilon)^{p-2} (R_1^\varepsilon)^p  f'(S^\varepsilon)\partial_t S^\eps \\ 
& = - pd_{S} \int_\Omega    f(S^\varepsilon)^{p-2} (R_1^\varepsilon)^{p-1} \nabla f(S^\varepsilon) \cdot \nabla R_1^\varepsilon  -  d_{S} \int_\Omega f(S^\varepsilon)^{p-2} f''(S^\varepsilon) (R_1^\varepsilon)^p |\nabla   S^\varepsilon |^2  \\
%%%%%%%%%%%
& - (p-2) d_{S}  \int_\Omega f(S^\varepsilon)^{p-3}     (R_1^\varepsilon)^{p} |\nabla f(S^\varepsilon)|^2  + \int_\Omega f(S^\varepsilon)^{p-2} f'(S^\varepsilon)  (R_1^\varepsilon)^p I(x,t)  \\
&\le \frac{pd_{R_1}      }{4} \int_\Omega f(S^\varepsilon)^{p-1} (R_1^\varepsilon)^{p-2} |\nabla R_1^\varepsilon|^2 +  \mu \eta_2  \int_\Omega f(S^\varepsilon)^{p-2} f'(S^\varepsilon) (R^\varepsilon)^{p+1}   \\
&+ \left( d_{S} + |p-2|d_{S} + \frac{ pd_S^2}{d_{R_2}} 
\right) \left( \sup_{\varepsilon>0} \left( \frac{|f'(S^\varepsilon)|^2}{f(S^\varepsilon)^{3-p}} +  \frac{|f''(S^\varepsilon)|}{f(S^\varepsilon)^{2-p}} \right) \right) \int_\Omega  (R_1^\varepsilon)^{p} |\nabla S^\varepsilon|^2  \\
&\le \frac{pd_{R_1}      }{4} \int_\Omega f(S^\varepsilon)^{p-1} (R_1^\varepsilon)^{p-2} |\nabla R_1^\varepsilon|^2 +  C_{p,T}  \int_\Omega (R^\varepsilon)^{p+1} + C_{p,T} \int_\Omega  (R_1^\varepsilon)^{p} |\nabla S^\varepsilon|^2 ,
\end{aligned}
\end{equation}
where we note that $I(x,t) :=\mu  (\eta_1R_1^\varepsilon +   \eta_2 R_2^\varepsilon ) - \rho S^\varepsilon \le \mu \eta_2 R^\varepsilon$. An estimate for the term containing $\partial_t g(S^\varepsilon)$ can be obtained similarly. 

\medskip

On the other hand,    
\begin{equation}
\begin{aligned}
& -  p d_{R_1} \int_\Omega f(S^\varepsilon)^{p-2} (R_1^\varepsilon)^{p-1} \nabla f(S^\varepsilon) \nabla R_1^\varepsilon -  p d_{R_2} \int_\Omega g(S^\varepsilon)^{p-2} (R_2^\varepsilon)^{p-1} \nabla g(S^\varepsilon) \cdot \nabla R_2^\varepsilon\\
& \hspace*{0.8cm} \le \frac{pd_{R_1}}{4} \int_\Omega f(S^\varepsilon)^{p-1} (R_1^\varepsilon)^{p-2} |\nabla R_1^\varepsilon|^2 + \left( \sup_{\varepsilon>0}  \frac{|f'(S^\varepsilon)|^2}{f(S^\varepsilon)^{3-p}} \right)  \int_\Omega  (R_1^\varepsilon)^{p}   | \nabla S^\varepsilon|^2 \\
& \hspace*{0.8cm} + \frac{pd_{R_2}}{4} \int_\Omega g(S^\varepsilon)^{p-1} (R_2^\varepsilon)^{p-2} |\nabla R_2^\varepsilon|^2 + \left( \sup_{\varepsilon>0}  \frac{|g'(S^\varepsilon)|^2}{g(S^\varepsilon)^{3-p}} \right)  \int_\Omega  (R_2^\varepsilon)^{p}   | \nabla S^\varepsilon|^2 \\
& \hspace*{0.8cm} \le \frac{pd_{R_1}}{4} \int_\Omega f(S^\varepsilon)^{p-1} (R_1^\varepsilon)^{p-2} |\nabla R_1^\varepsilon|^2 + C_{p,T} \int_\Omega  (R_1^\varepsilon)^{p}   | \nabla S^\varepsilon|^2 \\
& \hspace*{0.8cm} + \frac{pd_{R_2}}{4} \int_\Omega g(S^\varepsilon)^{p-1} (R_2^\varepsilon)^{p-2} |\nabla R_2^\varepsilon|^2 + C_{p,T} \int_\Omega  (R_2^\varepsilon)^{p}   | \nabla S^\varepsilon|^2. 
 \end{aligned}
\label{Calculate.EnergyStep3}  
\end{equation}
Finally, it is obvious that the sum of the last two terms in \eqref{Calculate.EnergyStep1} is bounded by 
$\frac{p \gamma_1 \Rhat}{p-1} E_p^\varepsilon(t)$ by skipping the negative parts. A combination of all estimates above gives 
\begin{align*}
  \frac{dE_p^\varepsilon}{dt} & + C_{d_{R_1},p} \int_\Omega   f(S^\varepsilon)^{p-1} (R_1^\varepsilon)^{p-2} |\nabla R_1^\varepsilon |^2 + C_{d_{R_2},p} \int_\Omega   g(S^\varepsilon)^{p-1} (R_2^\varepsilon)^{p-2} |\nabla R_2^\varepsilon |^2    \\
& +  \frac{1}{\eps} \int_\Omega  \Big( (f(S^\varepsilon) R_1^\varepsilon )^{p-1} -
\big( g(S^\varepsilon) R_2^\varepsilon \big)^{p-1} \Big) \Big( f(S^\varepsilon)   R_1^\varepsilon   - g(S^\varepsilon)  R_2^\varepsilon) \Big) \\
&\le C_{p,T} \left( E_p^\varepsilon +  \int_\Omega  (R_1^\varepsilon)^{p} |\nabla S^\varepsilon|^2  +  \int_\Omega (R^\varepsilon)^{p+1} \right) ,  
\end{align*}
and therefore the estimate \eqref{Estimate.Energy} by noting \eqref{Calculate.EnergyStep2a}. 
\end{proof}

With a glance at the energy estimate \eqref{Estimate.Energy}, we cannot take $p>1+\delta$ since the uniform regularity of $R^\varepsilon$ is just given in $L^{2+\delta}(\Omega_T)$ (Lemma \ref{Lem.UniformRegularity}). However, better feedback can be obtained via an appropriate interpolation.  Indeed, by Lemma \ref{Lem.UniformRegularity}, the uniform regularity of $\nabla S^\varepsilon$ is obtained in $L^q(\Omega_T)$ with $q>4+2\delta$. This will allow $p$ to be chosen strictly greater than $1+\delta$. This feedback will be presented precisely in the two next lemmas.

\begin{lemma}[Feedback from energy estimate] \label{Lem.Feedback} If there exists $\alpha>0$ such that  
\begin{align}
\sup_{\varepsilon>0}   \| R^\varepsilon \|_{L^{2+\alpha}(\Omega_T)} \le C_T, 
\label{Condition.Feedback}
\end{align}
then 
\begin{align*}
\sup_{\varepsilon>0}   \| R^\varepsilon \|_{L^{2+2\alpha}(\Omega_T)} \le C_{\alpha,T}.
\end{align*}
\end{lemma}

\begin{proof} With the uniform boundedness of $S^\varepsilon$ in $L^\infty(\Omega_T)$ as given in Lemma \ref{Lem.UniformRegularity}  and of $R^\varepsilon$ in $L^{2+\alpha}(\Omega_T)$ as \eqref{Condition.Feedback}, the quantity $\mu(\eta_1R_1^\varepsilon+\eta_2R_2^\varepsilon)-\rho S^\varepsilon$ is uniformly bounded in $L^{2+\alpha}(\Omega_T)$. Hence, from the equation for $S^\varepsilon$,  an application of the heat regularisation (Lemma \ref{Lem.HeatRegularisation})  yields   
$$\nabla S^\varepsilon \in L^{\frac{4(2+\alpha)}{2-\alpha}}(\Omega_T)^2 \subset L^{2(2+\alpha)}(\Omega_T)^2 .$$ 
By the H\"older inequality,  
\begin{align*}
\dintt{} (R^\varepsilon)^{1+\alpha} |\nabla S^\varepsilon|^2 
\le 
\left( \dintt{} (R^\varepsilon)^{2+\alpha} \right)^{\frac{1+\alpha}{2+\alpha}} \left( \dintt{} |\nabla S^\varepsilon|^{2(2+\alpha)} \right)^{\frac{1}{2+\alpha}} \le  C_{\alpha,T} .  
\end{align*}
Therefore, from the energy estimate \eqref{Estimate.Energy}, we get
\begin{gather*}
E_{p}^\varepsilon(t)    +  \dintt{}  |\nabla (R_1^\varepsilon)^{\frac{p}{2}} |^2 +  \dintt{} |\nabla (R_2^\varepsilon)^{\frac{p}{2}} |^2  +  \frac{1}{\eps} \dintt{}  Q_{p}^\varepsilon  
 \le C_{p,T,E_p(0)} \left(1  +  \int_0^t E_p^\varepsilon(s) \right),   
\end{gather*}
which holds for $p=1+\alpha$. Here, we note that the initial value $E_p(0)$ is finite due to Assumption \ref{Assumption.AllinOne}.    Now, the Gr\"onwall inequality can be applied to see that $\sup_{\varepsilon>0} E_p^\varepsilon(t)$ is  bounded on $(0,T)$. Consequently, we obtain 
\begin{gather*}
\|R^\varepsilon\|_{L^\infty(0,T; L^p(\Omega))}^p  +  \dint{} |\nabla (R_1^\varepsilon)^{\frac{p}{2}} |^2 +  \dint{} |\nabla (R_2^\varepsilon)^{\frac{p}{2}} |^2  
 \le  C_{p,T,E_p(0)},  
\end{gather*}
or more clearly, the terms  $(R_1^\varepsilon)^{p/2}, (R_2^\varepsilon)^{p/2}$ are uniformly bounded in $L^\infty(0,T;L^2(\Omega)) \cap L^2(0,T;H^1(\Omega))$.   
Thus, by the Gagliardo-Nirenberg  inequality,  
\begin{align*}
\| (R^\varepsilon)^{\frac{p}{2}} \|_{L^4(\Omega)} \le \| \nabla (R^\varepsilon)^{\frac{p}{2}} \|_{L^2(\Omega)}^{\frac{1}{2}} \| (R^\varepsilon)^{\frac{p}{2}} \|_{L^2(\Omega)}^{\frac{1}{2}}     ,
\end{align*}
which implies
\begin{align*}
\| (R^\varepsilon)^{\frac{p}{2}} \|_{L^{4}(\Omega_T)}^{4} \le  \| \nabla (R^\varepsilon)^{\frac{p}{2}} \|_{L^2(\Omega_T)}^2 \|R^\varepsilon\|_{L^\infty(0,T; L^p(\Omega))}^p \le C_{p,T},
\end{align*}
i.e., the uniform boundedness of $R^\varepsilon$ in  $L^{2+2\alpha}(\Omega_T)$ is obtained. 
\end{proof}

\begin{lemma}[Bootstrap argument] 
\label{Lem.Bootstrap} For any $1\le q<\infty$, it holds that 
\begin{align*}
\dint{} (R^\varepsilon)^{q} +   \dint{} \left(|\nabla R_1^\varepsilon |^2 + |\nabla R_2^\varepsilon |^2 \right) + \|S^\varepsilon\|_{W^{2,1}_q(\Omega_T)}  \le C_T.
\end{align*} 
In addition, the fast reaction term can be estimated as   
\begin{equation}
\|f(S^\varepsilon) R_1^\varepsilon   - g(S^\varepsilon) R_2^\varepsilon\|_{L^q(\Omega_T)} \le C_T \varepsilon^{1/q} \quad \text{for any } 2\le q<\infty.
\label{Estimate.ConvToCritMani}
\end{equation}
\end{lemma}

\begin{proof} By Lemma \ref{Lem.UniformRegularity},  $R^\varepsilon$ is uniformly bounded in $L^{2+\delta }(\Omega_T)$. Then, Lemma \ref{Lem.Feedback} gives the feedback that $R^\varepsilon$ is uniformly bounded in $L^{2+2\delta }(\Omega_T)$. An iteration of this feedback yields 
\begin{align*}
\sup_{\varepsilon>0}   \| R^\varepsilon \|_{L^{2+2^n\delta}(\Omega_T)} \le C_{\delta,n,T} 
\end{align*}
for any $n\in \mathbb{N}$, i.e.,   $R^\varepsilon$ is uniformly bounded in $L^q(\Omega_T)$ for any $1\le q<\infty$. This allows us to choose $p=2$ in the energy estimate \eqref{Estimate.Energy} to have 
\begin{align*}
\dint{}|\nabla R_1^\varepsilon |^2  + \dint{}|\nabla R_2^\varepsilon |^2 \le C_T.
\end{align*}
Note that, from Lemma \ref{Lem.EnergyEstimate} we also have 
\begin{align*}
\frac{1}{\eps} \dint{} Q_{p}^\varepsilon   \le C_{p,T}, \quad \text{for any } 2\le p <\infty.   
\end{align*}
By the basic inequality $|x^{p-1}-y^{p-1}|\ge |x-y|^{p-1}$ for all $x,y\ge0$ and $p\ge 2$, we directly imply the convergence \eqref{Estimate.ConvToCritMani}. 
 Finally, with the uniform boundedness of $S^\varepsilon$ in $L^\infty(\Omega_T)$ and of $R^\varepsilon$ in $L^q(\Omega_T)$, the quantity $\mu \big(\eta_1R_1^\varepsilon +  \eta_2 R_2^\varepsilon\big) - \rho S^\varepsilon $ 
is uniformly bounded in $L^q(\Omega_T)$ (for any $1\le q<\infty$). Therefore, an application of the heat regularisation, Lemma  \ref{Lem.HeatRegularisation}, to the equation for $S^\varepsilon$ yields the uniform boundedness of $S^\varepsilon$ in $W^{2,1}_q(\Omega_T)$. 
\end{proof}

\subsection{Fast reaction limit}
\label{Sec.FRLsmoothcase}

In Theorem \ref{Theo.ConvergeToLimitSys}, we will show that the $\varepsilon$-dependent system can be reduced to the limiting system \eqref{System.LimitingSys} combining with the algebraic equation \eqref{Formula.CriticalManifold}, where, in more details, $R^\varepsilon$, $\nabla R^\varepsilon$, $\nabla S^\varepsilon$ are uniformly bounded in $L^k(\Omega_T)$ for any $1\le k<\infty$, $L^2(\Omega_T)^2$, $L^\infty(\Omega_T)^2$ respectively. First, let us state the concept of a weak solution to the limiting system.

\begin{definition}[Weak solution to the limiting system]\label{Definition.WeakSolution} A pair of non-negative functions $(R,S)$ is called a weak solution to \eqref{System.LimitingSys} if for any $T>0$ 
\begin{gather*}
R\in L^{4+}(\Omega_T), \, \nabla R \in L^2(\Omega_T)^2,\, \pa_t R \in L^2(0,T;H^{-1}(\Omega)),\quad S\in W^{2,1}_{4+}(\Omega_T), 
\end{gather*}
  and     
\begin{equation}\label{System.WeakSolution}
	\begin{aligned}
	 &  \dint{} \partial_t R \varphi + \dint{}\nabla \bigg( \bigg( d_{R_1} - (d_{R_1}-d_{R_2})\frac{f(S)}{f(S)+g(S)} \bigg) R \bigg) \cdot \nabla \varphi  =  \dint{} h(R,S) \varphi ,   \\		 
	 %%%%%%%%%%%%%%%%%
		 & \hspace*{0.15cm} \dint{} \partial_tS \psi   + d_S\dint{} \nabla S \cdot  \nabla \psi = \mu \dint{} (\eta_1 \xi_1(R,S) + \eta_2 \xi_2(R,S))\psi - \rho \dint{} S\psi,  
		\end{aligned}
\end{equation} 
for all test functions $\varphi, \psi \in L^{2}(0,T;H^{1}(\Omega))$, where $h$ is given by \eqref{Formula.hRS}. 
\end{definition}

\begin{proof}[Proof of Theorem \ref{Theo.ConvergeToLimitSys}] Since the sequences $\{S^\varepsilon\}$ and $ \{\partial_t S^\varepsilon\}$ are respectively bounded in $L^\infty(0, T ;W^{1,\infty}(\Omega))$ and $L^s(\Omega_T)$, cf. Lemmas \ref{Lem.HeatRegularisation}, \ref{Lem.Bootstrap}, the Aubin-Lions lemma yields that  $\{S^\varepsilon\} $ is relatively compact in $\LQ{\infty}$. Thus, up to a subsequence (not relabeled),  
\begin{align}
S^\varepsilon \to  S \quad \text{in } L^{\infty}(\Omega_T) .
\label{Proof.TheoFSL.Step1}
\end{align}
On the other hand,   $\{\nabla R_1^\varepsilon\}$ and $\{\nabla R_2^\varepsilon\}$ are bounded in $L^{2}(\Omega_T)^2$ due to the bootstrap argument in Lemma \ref{Lem.Bootstrap}. Hence, it follows from adding the equations for $R_1^\varepsilon,R_2^\varepsilon$, i.e. 
\begin{align}
\partial_t R^\varepsilon = \sum_{i=1}^2 \Big( d_{R_i} \Delta R_i^\varepsilon  + \gamma_i R_i^\varepsilon(\Rhat-R^\varepsilon)  - \eta_iR_i^\varepsilon \Big),
\label{Equation.R} 
\end{align}
that $\{\partial_tR^\varepsilon\}$ is bounded in $L^2(0,T;(H^1(\Omega))^*)$. By applying the Aubin-Lions lemma again, we see from Lemma \ref{Lem.Bootstrap} that $R^\varepsilon$ is relatively compact in $L^2(\Omega_T)$, which can be improved up to $L^k(\Omega_T)$ for any $1\le k<\infty$ upon Lemma \ref{Lem.UniformRegularity}.  
Therefore, by the nonnegativity of $R_1^\varepsilon,R_2^\varepsilon$, we deduce
\begin{align}
R_{1}^\varepsilon \to R_1, \quad  R_{2}^\varepsilon \to  R_{2} \quad  \text{in } \LQ{k}, \quad 1\le k<\infty,
\label{Proof.TheoFSL.Step2}
\end{align}
up to a subsequence.  

\medskip

Next, we will prove that $(R,S)$ is a weak solution to \eqref{System.LimitingSys}, which is defined according to Definition \ref{Definition.WeakSolution}. 
First, one can observe that $f(S^\varepsilon)$  converges strongly to $f(S)$ in $\LQ{\infty}$. Indeed, by Lagrange’s mean value theorem, there exists $\widetilde{S}^\varepsilon=\widetilde{S}^\varepsilon(x,t)$ such that
\begin{align*}
\min\{S^\varepsilon;S\} \le \widetilde{S}^\varepsilon \le \max\{S^\varepsilon;S\}, 
\end{align*}
and
\begin{equation}
\begin{aligned}
& |f(S^\varepsilon)-f(S)| = |f'(\widetilde{S}^\varepsilon)| |S^\varepsilon-S| \le \left(\sup_{ \min\{S^\varepsilon;S\} 
 \le s \le  \max\{S^\varepsilon;S\} }  |f'(s)|\right) |S^\varepsilon-S|    
\end{aligned} 
\label{Proof.TheoFSL.Step3}
\end{equation}
since $f$ is a $C^2$ function as Assumption \ref{Assumption.AllinOne}. The same argument shows that $g(S^\varepsilon)$  converges strongly to $g(S)$ in $\LQ{\infty}$. Secondly, direct computations for  
\begin{equation*}
\begin{aligned}
&K_1^\varepsilon:= \frac{f(S^\varepsilon)}{f(S^\varepsilon)+g(S^\varepsilon)} - \frac{f(S)}{f(S)+g(S)} , \\
%%%%%%%%%%%%%%%%%
&K_2^\varepsilon:=\frac{g(S^\varepsilon)}{f(S^\varepsilon)+g(S^\varepsilon)} - \frac{g(S)}{f(S)+g(S)},    
\end{aligned}
\end{equation*}
show
\begin{equation}
\label{Proof.TheoFSL.Step4}
\begin{aligned}
&K_1^\varepsilon = \frac{g(S)}{f(S)+g(S)} \frac{f(S^\varepsilon) - f(S)}{f(S^\varepsilon) + g(S^\varepsilon)} - \frac{f(S)}{f(S)+g(S)}   \frac{g(S^\varepsilon)-g(S)}{f(S^\varepsilon) + g(S^\varepsilon)} , \\
%%%%%%%%%%%%%%%%%
&K_2^\varepsilon = \frac{f(S)}{f(S)+g(S)} \frac{g(S^\varepsilon) - g(S)}{f(S^\varepsilon) + g(S^\varepsilon)} - \frac{g(S)}{f(S)+g(S)}   \frac{f(S^\varepsilon)-f(S)}{f(S^\varepsilon) + g(S^\varepsilon)} .   
\end{aligned}
\end{equation}
Since $f(S^\varepsilon)$,  $g(S^\varepsilon)$ have positive bounds from below in \eqref{Calculate.EnergyStep2a},
\begin{equation}
\begin{aligned}
& \frac{f(S^\varepsilon)}{f(S^\varepsilon)+g(S^\varepsilon)} \to \frac{f(S)}{f(S)+g(S)}  \quad \text{in } \LQ{\infty}, \\
& \frac{g(S^\varepsilon)}{f(S^\varepsilon)+g(S^\varepsilon)} \to \frac{g(S)}{f(S)+g(S)} \quad \text{in } \LQ{\infty}. 
\label{Proof.TheoFSL.Step5}
\end{aligned}
\end{equation}
Now, by the uniform lower boundedness of $f(S^\varepsilon)$,  $g(S^\varepsilon)$ again,  
\begin{align}
 \left\|R_1^\varepsilon - \frac{g(S^\varepsilon)}{f(S^\varepsilon)+g(S^\varepsilon)} R^\varepsilon \right\|_{L^2(\Omega_T)} \le \,& C_T  \| f(S^\varepsilon) R_1^\varepsilon   - g(S^\varepsilon) R_2^\varepsilon \|_{L^2(\Omega_T)}   , \nonumber\\
 \left\|R_2^\varepsilon - \frac{f(S^\varepsilon)}{f(S^\varepsilon)+g(S^\varepsilon)}R^\varepsilon \right\|_{L^2(\Omega_T)} \le \,& C_T \| f(S^\varepsilon) R_1^\varepsilon   - g(S^\varepsilon) R_2^\varepsilon \|_{L^2(\Omega_T)} . \label{Proof.TheoFSL.Step6}
\end{align}
Then, by estimating the fast reaction term in Lemma \ref{Lem.Bootstrap}, we can use the triangle inequality to imply from \eqref{Proof.TheoFSL.Step2} and  \eqref{Proof.TheoFSL.Step3} that, for any $1\le k<\infty$,
\begin{equation} 
\begin{aligned}
R_1^\varepsilon \to \frac{g(S)}{f(S)+g(S)} R=\xi_1(R,S) \quad \text{in } \LQ{k},\\
 R_2^\varepsilon \to \frac{f(S)}{f(S)+g(S)} R=\xi_2(R,S) \quad \text{in } \LQ{k},
\end{aligned}
\label{Proof.TheoFSL.Step7}
\end{equation}
which also yields the following weak convergence
\begin{gather*}
\nabla( d_{R_1} R_1^\varepsilon + d_{R_2} R_2^\varepsilon)  
\rightharpoonup \nabla  \left( \bra{d_{R_1} - (d_{R_1}-d_{R_2})\frac{f(S)}{f(S)+g(S)}}R \right)  \quad \text{in } L^{2}(\Omega_T). 
\end{gather*}

Now, by adding the equations for $R_1^\varepsilon$ and $R_2^\varepsilon$  we have 
\begin{align*} 
& \dint{} \partial_t R^\varepsilon \varphi + \dint{}\nabla (d_{R_1} R_1^\varepsilon + d_{R_2} R_2^\varepsilon) \cdot \nabla \varphi \\
& \hspace*{3.3cm} = \dint{} \Big( (\gamma_1 R_1^\varepsilon + \gamma_2 R_2^\varepsilon)(\Rhat-R^\varepsilon)  -  (\eta_1 R_1^\varepsilon+\eta_2 R_2^\varepsilon) \Big) \varphi , \\
	 %%%%%%%%%%%%%%
& \dint{} \partial_tS^\varepsilon \psi   + d_S\dint{} \nabla S^\varepsilon \cdot \nabla \psi = \mu \dint{} (\eta_1 R_1^\varepsilon + \eta_2 R_2^\varepsilon )\psi - \rho \dint{} S^\varepsilon \psi ,
\end{align*}
for all test functions $\varphi, \psi \in L^{2}(0,T;H^{1}(\Omega))$, which after sending $\varepsilon$ to $0$ shows that  $(R,S)$ is a weak solution to the limiting system  \eqref{System.LimitingSys}. Finally,  $S$ is, in fact, a strong solution to the problem  
	\begin{equation}
\begin{aligned}
\left\{
	\begin{array}{rlll}
		\displaystyle \pa_t S - d_S\Delta S &\hspace*{-0.2cm}=\hspace*{-0.2cm}& \mu \big(\eta_1\xi_1 +  \eta_2 \xi_2 \big) - \rho S 
 & \text{in } \Omega_T, \\
	\nabla S \cdot \nu  &\hspace*{-0.2cm}=\hspace*{-0.2cm}& 0 & \text{on } \Gamma_T, \\
	S(0)&\hspace*{-0.2cm}=\hspace*{-0.2cm}& S_0 & \text{on } \Omega,   
	\end{array}
	\right.
\end{aligned}
\label{Problem.S}	
\end{equation}
where  $R \in L^k(\Omega_T)$ for any $1\le k<\infty$ and 
\begin{align*}
\xi_1 = \frac{g(S)}{f(S)+g(S)} R \in \LQ{k} ,  \quad 
\xi_2 = \frac{f(S)}{f(S)+g(S)} R \in \LQ{k}.
\end{align*}
By the heat regularisation in Lemma \ref{Lem.HeatRegularisation},   $S\in W^{2,1}_{k}(\Omega_T)$ and particularly $S,|\nabla S|\in L^\infty(\Omega_T)$. 
\end{proof}

\section{On the limiting system}\label{sec3}

\subsection{$L^\infty$-boundedness via Alikakos iteration}

We will improve the regularity of weak solutions to the limiting system by performing an Alikakos iteration \cite{alikakos1979application}, which is presented in Lemma \ref{Lem.LInfinityEstimate} with an $L^\infty(\Omega_T)$-bound for the component $R$.

\begin{lemma}
[$L^\infty$-estimate for the roots] 
\label{Lem.LInfinityEstimate}
Let $(R,S)$ be a weak solution to the limiting system \eqref{System.LimitingSys}. Then 
\begin{align}
\|R\|_{L^\infty(\Omega_T)} + \|S\|_{L^\infty(\Omega_T)} + \|\nabla S\|_{L^\infty(\Omega_T)^2} \le C_T. 
\label{Estimate.LInfinity}
\end{align}
\end{lemma}
\begin{proof}  First, regularity of $S$ is clear due to the Sobolev embedding. It is only necessary to prove the regularity of   $R$.  

\medskip

\noindent {\bf Estimate on the total roots}: Taking the test function in \eqref{System.WeakSolution} by $\frac{1}{s}\chi_{(t,t+s)}$, for $t\in(0,T)$ and $s>0$, gives $
\partial_t \|R\|_{L^1(\Omega)} = \int_\Omega h(R,S) $ and hence
\begin{align*}
\partial_t \|R(t)\|_{L^1(\Omega)}  & \le  \|R(t)\|_{L^1(\Omega)} \left(\Rhat \max(\gamma_1;\gamma_2) - \frac{\min(\gamma_1;\gamma_2)}{ |\Omega|}   \|R(t)\|_{L^1(\Omega)} \right) . 
\end{align*}
Consequently, this yields  
\begin{align*}
\|R\|_{L^\infty(0,T;L^1(\Omega))} \le C( \gamma_1,\gamma_2,|\Omega|,\|R_0\|_{L^1(\Omega)}). 
\end{align*}
	
\noindent {\bf $L^p$- Estimate}: 
To obtain $L^p$-estimates  for  $R$, we follow the Alikakos iteration, \cite{alikakos1979application}. Let  $p_n=2^n$ for all $n\in \mathbb{N}$. Then $R^{p_n-1}\in L^4(\Omega_T)$ and $\nabla R^{p_n-1}\in L^2(\Omega_T)^2$ for $n=1$. So, by the induction hypothesis, for $n\ge 1$ we can take $R^{p_{n}-1}(\frac{1}{s}\chi_{(t,t+s)})$, for $t\in(0,T)$ and $s>0$, to be a test function in  \eqref{System.WeakSolution}, which gives
\begin{equation}
\begin{gathered}
\frac{1}{p_{n}} \partial_t \int_\Omega  (R^{p_{(n-1)}})^2   + \frac{4(p_{n}-1)}{p_{n}^2} \int_\Omega  \bigg( d_{R_1} - (d_{R_1}-d_{R_2})\frac{f(S)}{f(S)+g(S)} \bigg) |\nabla R^{p_{(n-1)}}|^2 \\
 =   \frac{(p_{n}-1)(d_{R_1}-d_{R_2})}{p_{(n-1)}} \int_\Omega  R^{p_{(n-1)}} \nabla \bigg(    \frac{f(S)}{f(S)+g(S)} \bigg) \cdot   \nabla R^{p_{(n-1)}} + \int_\Omega  h(R,S) R^{p_{n}-1} . 
\end{gathered}
\label{Proof.LemLInfinityStep1}
\end{equation}
Note that, by Assumption \eqref{Assumption.InitialData}, we can argue similarly as Lemma \ref{Lem.LowerBound} to see that $S\ge s_0>0$ on $\Omega_T$ in which $S$ is in fact the strong solution to \eqref{Problem.S}. Hence, in the same way as \eqref{Calculate.EnergyStep2a}, we can claim that $f(S)+g(S)$ has a positive bound from below. Subsequently, by Assumption \ref{Assumption.AllinOne},
\begin{equation}
\begin{aligned}
& \bigg|\nabla \bigg(    \frac{f(S)}{f(S)+g(S)} \bigg)\bigg| \le \frac{|f'(S)|g(S)+f(S)| g'(S) |}{(f(S)+g(S))^2} |\nabla S| \\
&   \le \left( \sup_{s_0\le s\le \|S\|_{L^\infty(\Omega_T)}} \frac{|f'(s)|g(s)+f(s)| g'(s) |}{(f(s)+g(s))^2} \right) \|\nabla S\|_{L^\infty(\Omega_T)} .
\end{aligned}
\label{Proof.LemLInfinityStep2}
\end{equation} 
Then, it follows from \eqref{Proof.LemLInfinityStep1} that 
\begin{equation*}
\begin{aligned}
& \partial_t \int_\Omega  (R^{p_{(n-1)}})^2  + \frac{4d_{R_2}(p_{n}-1)}{p_{n}} \int_\Omega  |\nabla R^{p_{(n-1)}}|^2 \\
& \le C_{1,T} \frac{p_{n}(p_{n}-1)}{p_{(n-1)}}  \int_\Omega  R^{p_{(n-1)}} |\nabla R^{p_{(n-1)}}| + C_2p_{n} \int_\Omega     (R^{p_{(n-1)}})^2 , 
\end{aligned}
\end{equation*}
by observing 
$$ d_{R_1} - (d_{R_1}-d_{R_2}) \frac{f(S)}{f(S)+g(S)}  \ge d_{R_2}, \,\,\, \text{and} \,\,\, h(R,S) \le C_2R$$ with a $T$-independent constant $C_2>0$.  This deduces 
\begin{equation}
\begin{gathered}
\partial_t\int_\Omega  (R^{p_{(n-1)}})^2 + d_{R_2} \dintt{}  |\nabla R^{p_{(n-1)}}|^2  \le   a_{n} \int_\Omega  (R^{p_{(n-1)}})^2 , 
\end{gathered} 
\label{Proof.LemLInfinityStep3}
\end{equation}
where we denote
$a_{n}:= C_{1,T}^2p_{n}(p_{n}-1)/(8d_{R_2}) + C_2p_{n}$. 
Let $\sigma_n\in(0,1)$ be  small enough such that $
(a_{n}+\sigma_{n})\sigma_{n} \le d_{R_2}.$ 
By the Gagliardo-Nirenberg and Young inequalities, we have   
\begin{align*}
\int_\Omega (R^{p_{(n-1)}})^2 \le \sigma_n \int_\Omega |\nabla R^{p_{(n-1)}}|^2 + \frac{C_{3,\Omega}}{\sigma_n} \left( \int_\Omega R^{p_{(n-1)}} \right)^2,
\end{align*}
and the differential inequality \eqref{Proof.LemLInfinityStep3}  implies
\begin{equation*}
\begin{gathered}
\int_\Omega  R^{2^n} \le \max \left\{ C_{3,\Omega} \frac{a_n+\sigma_n}{\sigma_n^2} \left( \int_\Omega R^{2^{(n-1)}} \right)^2 ; \, \|R_0\|_{L^\infty(\Omega)}^{2^n} \right\}. 
\end{gathered} 
\end{equation*}
For simplicity, we can set $C_{3,\Omega}=1$. Since $(a_n+\sigma_n)/\sigma_n^2\ge 1$, an iteration yields
\begin{equation*}
\begin{aligned}
\int_\Omega  R^{2^n} &\le \prod_{i=1}^1 \left( \frac{a_{n+1-i}+\sigma_{n+1-i}}{\sigma_{n+1-i}^2} \right)^{2^{i-1}} \left( \max \left\{ \left( \int_\Omega R^{2^{(n-1)}} \right)^{2^1} ; \, \|R_0\|_{L^\infty(\Omega)}^{2^n} \right\} \right) \\
%%%%%%%%%%%%%%%%%%%%%
&\le \prod_{i=1}^2 \left( \frac{a_{n+1-i}+\sigma_{n+1-i}}{\sigma_{n+1-i}^2} \right)^{2^{i-1}} \left( \max \left\{ \left( \int_\Omega R^{2^{(n-2)}} \right)^{2^2} ; \, \|R_0\|_{L^\infty(\Omega)}^{2^n} \right\} \right)\\
&...\\
&\le \prod_{i=1}^n \left( \frac{a_{n+1-i}+\sigma_{n+1-i}}{\sigma_{n+1-i}^2} \right)^{2^{i-1}} \left( \max \left\{ \|R\|_{L^\infty(0,T;L^1(\Omega))}^{2^n} ; \, \|R_0\|_{L^\infty(\Omega)}^{2^n} \right\} \right).
\end{aligned} 
\end{equation*}
Note that $a_n\lesssim  2^{2n}$, hence we can choose $\sigma_n \sim 2^{-2n}$ to have  $(a_n+\sigma_n)/\sigma_n^2 \lesssim 2^{6n}$ (as well as $(a_n+\sigma_n)\sigma_n \le d_{R_2}$) for all $n\ge 1$. Here, the notation $\alpha_n\lesssim \beta_n$ and $\alpha_n\lesssim \beta_n$ are used instead of $\alpha_n \le C \beta_n$ and $\alpha_n = C \beta_n$ for some independent constants $C$ of $n$. We then have   
\begin{align*}
\prod_{i=1}^n \left( \frac{a_{n+1-i}+\sigma_{n+1-i}}{\sigma_{n+1-i}^2} \right)^{2^{i-1}} \lesssim   64^{\left( \displaystyle  \sum_{i=1}^n (n+1-i)2^{i-1} \right)} \le 4096^{2^{n}}.  
\end{align*}
Finally, 
\begin{equation*}
\begin{aligned}
\| R^{2^n} \|_{L^\infty(0,T;L^{2^n}(\Omega))}  \lesssim  \max \left\{ \|R\|_{L^\infty(0,T;L^1(\Omega))}  ; \, \|R_0\|_{L^\infty(\Omega)}  \right\} ,
\end{aligned} 
\end{equation*}
which shows \eqref{Estimate.LInfinity} after sending $n \to \infty$.  
\end{proof}
 
\subsection{Uniqueness and stability of the weak solution}

In this part, we will show the uniqueness and stability of the weak solution by studying the difference between two weak solutions $(R,S)$, $(\widetilde R, \widetilde S)$ in $L^2(0,T;H^1(\Omega))$. For this purpose, we will subtract corresponding equations in the weak formulation, given in  Definition \ref{Definition.WeakSolution}, and choose the test functions $(\varphi,\psi)=(\widetilde R-R, \widetilde S-S)$. The most difficulty comes from estimating the term 
\begin{align}
\dintt{} \bigg( \frac{f(\widetilde S)}{f(\widetilde S)+g(\widetilde S)} - \frac{f(S)}{f(S)+g(S)} \bigg) \nabla R \cdot \nabla (\widetilde{R}-R) \label{Formula.DiffTerm}
\end{align}
since $\nabla \widetilde R, \nabla R$ just belong to $L^2(\Omega_T)^2$. To overcome, in the following lemma we will first control $\|\widetilde S(t) - S(t)\|_{L^\infty(\Omega)}$ by the norm of $(\widetilde R-R, \widetilde S-S)$ in $L^2(0,t;H^1(\Omega))$ by using the heat semigroup and interpolation inequalities, and then estimate the term \eqref{Formula.DiffTerm} in Lemma \ref{Lem.DifficultTerm}.    

\begin{lemma} 
\label{Lem.LInfDifference}
Assume $(R,S), (\widetilde{R},\widetilde{S})$ are weak solutions to the system \eqref{System.WeakSolution} corresponding to the initial data $(R_0,S_0)$ and $(\widetilde R_0,\widetilde S_0)$, which satisfy Assumption \eqref{Assumption.InitialData}. 
Then 
 \begin{equation}
\begin{aligned}
  \|\widetilde S(t) - S(t)\|_{L^\infty(\Omega)}^2  
  &\le C \|\widetilde S_0 - S_0\|_{W^{2,1}_{2+}(\Omega)}^2 + C_{\alpha_1,T} \dintt{}   \Big(  |\widetilde R-R |^2 + |\widetilde S-S|^2 \Big)  
  \\
  &+ \alpha_1 \dintt{}   \Big( | \nabla(\widetilde R-R)|^2 + | \nabla(\widetilde S-S)|^2 \Big) ,  
\end{aligned}
\end{equation}
for all $t\in(0,T)$ and  any $\alpha_1>0$, where $C_{\alpha_1,T}$ is a finite constant depending on $\alpha_1$. 
\end{lemma}

\begin{proof} The lower boundedness of $S$, $\widetilde{S}$ by a positive constant, and then of $f(S)+g(S)$, $f(\widetilde S)+g(\widetilde S)$, is followed similarly as \eqref{Calculate.EnergyStep2a}. By the same computation as  \eqref{Proof.TheoFSL.Step4},
\begin{equation}
\begin{aligned}
\bigg| \frac{f(\widetilde S)}{f(\widetilde S)+g(\widetilde S)} - \frac{f(S)}{f(S)+g(S)} \bigg|  \le C_{T}  \left( |f(\widetilde S) - f(S)| + |g(\widetilde S) - g(S)| \right),\\
 \bigg| \frac{g(\widetilde S)}{f(\widetilde S)+g(\widetilde S)} - \frac{g(S)}{f(S)+g(S)} \bigg|  \le C_{T}  \left( |f(\widetilde S) - f(S)| + |g(\widetilde S) - g(S)| \right).
\end{aligned}
\label{Proof.LInfDifferenceStep1}
\end{equation}
Here, due to the Lagrange’s mean value theorem, similarly as \eqref{Proof.TheoFSL.Step3},  
\begin{align}
|f(\widetilde S) - f(S) | + |g(\widetilde S) - g(S) | \le C_{T} |\widetilde S - S | .
\label{Proof.LInfDifferenceStep2}
\end{align} 
Moreover, straightforward computations show
\begin{align*}
& |\xi_1(\widetilde R,\widetilde S)-\xi_1(R,S)| + |\xi_2(\widetilde R,\widetilde S)-\xi_2(R,S)| \\
& = \bigg| \frac{g(\widetilde S)}{f(\widetilde S)+g(\widetilde S)} (\widetilde R-R) + \left( \frac{g(\widetilde S)}{f(\widetilde S)+g(\widetilde S)} - \frac{g(S)}{f(S)+g(S)} \right) R \bigg| \\
& + \bigg| \frac{f(\widetilde S)}{f(\widetilde S)+g(\widetilde S)} (\widetilde R-R) + \left( \frac{f(\widetilde S)}{f(\widetilde S)+g(\widetilde S)} - \frac{f(S)}{f(S)+g(S)} \right) R \bigg| \\
& \le   |\widetilde R - R| + C_{T} \big( |f(\widetilde S)-f(S)| + |g(\widetilde S)-g(S)| \big) R, 
\end{align*}
and therefore
\begin{align}
  |\xi_1(\widetilde R,\widetilde S)-\xi_1(R,S)| + |\xi_2(\widetilde R,\widetilde S)-\xi_2(R,S)| 
  \le C_T  \left( |\widetilde R - R| +|\widetilde S-S| \right) 
\label{Proof.LInfDifferenceStep5}
\end{align}
by taking the $L^\infty$-boundedness of $R$ from Lemma  \ref{Lem.LInfinityEstimate} and the estimate  \eqref{Proof.LInfDifferenceStep2} together.   

\medskip

Now, noting that the difference $V=\widetilde S-S$ is a strong solution to the equation 
\begin{align}
\pa_t V - d_S\Delta V = \displaystyle   \mu \eta_1(\xi_1(\widetilde R,\widetilde S)-\xi_1(R,S)) +  \mu  \eta_2 (\xi_2(\widetilde R,\widetilde S)-\xi_2(R,S))   -\rho V
\label{Equation.S}
\end{align}
with the homogeneous Neumann boundary condition and the initial condition $\widetilde S_0 - S_0$. By using $L^p-L^q$ estimates for the heat semigroup, we have 
\begin{align}
\|e^{(\Delta-\rho)(t-\tau)} (\widetilde S_0 - S_0) \|_{L^\infty(\Omega)} \le  \|\widetilde S_0 - S_0\|_{L^\infty(\Omega)} \le C \|\widetilde S_0 - S_0\|_{W^{2,1}_{2+}(\Omega)},
\label{Proof.LInfDifferenceStep3}
\end{align}
and 
\begin{equation}
\begin{aligned}
& \bigg\| \int_0^t e^{(\Delta-\rho)(t-\tau)} \Big( \mu \eta_1(\xi_1(\widetilde R,\widetilde S)-\xi_1(R,S)) +  \mu  \eta_2 (\xi_2(\widetilde R,\widetilde S)-\xi_2(R,S)) \Big) \bigg\|_{L^\infty(\Omega)}^2 \\
%%%%%%%%%%%
&\le C  \left( \int_0^t  (t-\tau)^{-\frac{1}{2+\beta}}    \||\xi_1(\widetilde R,\widetilde S)-\xi_1(R,S)|+ |\xi_2(\widetilde R,\widetilde S)-\xi_2(R,S)|\|_{L^{2+\beta}(\Omega)}   \right)^2 \\
%%%%%%%%%%%
&\le C_{\beta}   T^{\frac{\beta}{2+\beta}} \int_0^t    \||\xi_1(\widetilde R,\widetilde S)-\xi_1(R,S) | + |\xi_2(\widetilde R,\widetilde S)-\xi_2(R,S)|\|_{L^{2+\beta}(\Omega)}^2 ,  
\end{aligned}
\label{Proof.LInfDifferenceStep4}
\end{equation} 
for small $\beta>0$ and $t\in(0,T)$. Therefore, it follows from variation of constant formula and the estimates  \eqref{Proof.LInfDifferenceStep5}-\eqref{Proof.LInfDifferenceStep3} that 
\begin{equation}
\begin{aligned}
  \|\widetilde S - S\|_{L^\infty(\Omega)}^2  
  &\le C \|\widetilde S_0 - S_0\|_{W^{2,1}_{2+}(\Omega)}^2 + C_{\beta,T} \int_0^t   \Big( \|\widetilde R-R \|_{L^{2+\beta}(\Omega)}^2 + \|\widetilde S-S\|_{L^{2+\beta}(\Omega)}^2 \Big) \\
  &\le C \|\widetilde S_0 - S_0\|_{W^{2,1}_{2+}(\Omega)}^2 +  C_{\alpha_1,\beta,T} \int_0^t   \Big( \|\widetilde R-R \|_{L^{2}(\Omega)}^2 + \|\widetilde S-S\|_{L^{2}(\Omega)}^2 \Big)  
  \\
  &+ \alpha_1 \int_0^t   \Big( \| \nabla(\widetilde R-R) \|_{L^{2}(\Omega)}^2 + \| \nabla(\widetilde S-S)\|_{L^{2}(\Omega)}^2 \Big)  
\end{aligned}
\label{Proof.LInfDifferenceStep6} 
\end{equation}  
for any $\alpha_1>0$, where    
\begin{gather*}
  \|\widetilde R-R \|_{L^{2+\beta}(\Omega)}^2  \le C_{\alpha_1,\beta,T}   \|\widetilde R-R \|_{L^{2}(\Omega)}^2 + (\alpha_1/C_{\beta,T})   \|\nabla(\widetilde R-R) \|_{L^{2}(\Omega)}^2, \\
  \|\widetilde S - S\|_{L^{2+\beta}(\Omega)}^2 \le C_{\alpha_1,\beta,T} \|\widetilde S - S\|_{L^{2}(\Omega)}^2 + (\alpha_1/C_{\beta,T})  \|\nabla(\widetilde S - S) \|_{L^{2}(\Omega)}^2,
\end{gather*}  
due to applications of the Gagliardo-Nirenberg and Young inequalities. 
\end{proof}

%\begin{lemma}[Lemma 3.3, Chapter II, \cite{ladyzhenskaia1988linear}]
%\label{Lem.SobolevEmbedding}
%Let $\Omega$ be bounded domain with sufficiently smooth boundary in $\mathbb{R}^N$. Then, for any $t>0$
%\begin{align*}
%\|\partial_t^{\,k} \partial_x^{\,l} v\|_{L^p(\Omega_t)} \le C  t^{1-k-\frac{l}{2}-\frac{N+2}{2}(\frac{1}{q}-\frac{1}{p})} \|v\|_{W^{2,1}_q(\Omega_t)} + Ct^{-k+\frac{l}{2}-\frac{N+2}{2}(\frac{1}{q}-\frac{1}{p})} \|v\|_{L^q(\Omega_t)}, 
%\end{align*} 
%in which 
%\begin{align*}
%p\ge q \quad \text{and} \quad 1-k - \frac{l}{2} - \frac{N+2}{2}\bigg(\frac{1}{q}-\frac{1}{p}\bigg) \ge 0,
%\end{align*}
%where the constant $C$ does  not depend on $t$. 
%\end{lemma}

\begin{lemma}
\label{Lem.DifficultTerm}
 Let $(R,S), (\widetilde{R},\widetilde{S})$ be given by Lemma \ref{Lem.LInfDifference}.  
 Then, for all $t\in (0,T)$,
\begin{align*}
& \dintt{} \bigg( \frac{f(\widetilde S)}{f(\widetilde S)+g(\widetilde S)} - \frac{f(S)}{f(S)+g(S)} \bigg) \nabla R \cdot \nabla (\widetilde{R}-R) \\
& \le   C_T \|\widetilde S_0 - S_0\|_{W^{2,1}_{2+}(\Omega)}^2 +  \frac{d_{R_2}}{5} \dintt{}  |\nabla (\widetilde{R}-R)|^2   + C_{T} \dintt{} \left( |\widetilde R-R |^2 + |\widetilde S-S |^2 \right) . 
\end{align*} 
\end{lemma}
\begin{proof} First, by using   \eqref{Proof.LInfDifferenceStep1}-\eqref{Proof.LInfDifferenceStep2}, we can estimate  $|f(\widetilde S)/(f(\widetilde S)+g(\widetilde S))-f(S)/(f(S)+g(S))|$ by $C_T|\widetilde S -S|$. Moreover,  the norm of $\widetilde S(t) -S(t)$ in $L^\infty(\Omega)$ is estimated by Lemma \ref{Lem.LInfDifference}. Therefore, by taking the boundedness of $\nabla R$ in $L^2(\Omega_T)^2$, cf. Definition \ref{Definition.WeakSolution}, we have 
\begin{align*}
& \dintt{} \bigg( \frac{f(\widetilde S)}{f(\widetilde S)+g(\widetilde S)} - \frac{f(S)}{f(S)+g(S)} \bigg) \nabla R \cdot \nabla (\widetilde{R}-R) \\
& \le C_{\alpha_2,T}  \|\nabla R\|_{L^2(\Omega_t)^2}^2 \left( \esssup_{\tau\in(0,t)} \|\widetilde S(\tau) - S(\tau)\|_{L^\infty(\Omega)}^2 \right)  + \alpha_2 \dintt{} |\nabla (\widetilde{R}-R)|^2 \\
& \le C_{\alpha_2,T} \|\widetilde S_0 - S_0\|_{W^{2,1}_{2+}(\Omega)}^2 +      C_{\alpha_1,\alpha_2,T} \dintt{} \left( |\widetilde R-R |^2 + |\widetilde S-S |^2 \right) \\
& + \alpha_1 C_{\alpha_2,T} \dintt{} \left( |\nabla (\widetilde{R}-R)|^2 +  |\nabla(\widetilde S-S)|^2 \right)   +  \alpha_2 \dintt{} |\nabla (\widetilde{R}-R)|^2 ,
\end{align*}
for any $\alpha_1>0$,  $\alpha_2>0$. Then, by choosing these parameters  small enough, 
\begin{align*}
& \dintt{} \bigg( \frac{f(\widetilde S)}{f(\widetilde S)+g(\widetilde S)} - \frac{f(S)}{f(S)+g(S)} \bigg) \nabla R \cdot \nabla (\widetilde{R}-R) \\
& \le C_{d_{R_2},T} \|\widetilde S_0 - S_0\|_{W^{2,1}_{2+}(\Omega)}^2  +  \frac{1}{5}d_{R_2} \dintt{}  |\nabla (\widetilde{R}-R)|^2 \\
&  +  C_{T}  \dintt{} |\nabla(\widetilde S-S)|^2 + C_{T} \dintt{} \left( |\widetilde R-R |^2 + |\widetilde S-S |^2 \right) . 
\end{align*}
To obtain the desired estimate, we need to estimate the latter, including $\nabla(\widetilde S-S)$, which can be done as follows upon multiplying the equation \eqref{Equation.S} by $\widetilde S-S$   
\begin{equation}
\begin{aligned}
&  d_S\dintt{} |\nabla (\widetilde S - S)|^2
 \\
& = \dintt{} \mu \left( \sum_{i=1}^2 \eta_i(\xi_i(\widetilde R,\widetilde S)-\xi_i(R,S)) \right) (\widetilde S - S) -  \dintt{} \partial_t|\widetilde S-S|^2 \\
& \le  \mu \eta_2 C_T \dintt{}  \left( |\widetilde R - R| +|\widetilde S-S| \right)  |\widetilde S - S| - \int_\Omega |\widetilde S-S|^2 + \int_\Omega |\widetilde S_0-S_0|^2 \\
& \le  \mu \eta_2 C_{T} \dintt{} \left( |\widetilde R-R |^2 + |\widetilde S-S |^2 \right) - \int_\Omega |\widetilde S-S|^2 + C \|\widetilde S_0 - S_0\|_{W^{2,1}_{2+}(\Omega)}^2,  
\end{aligned}
\label{Proof.LemDifficultTerm}
\end{equation}
where we have used \eqref{Proof.LInfDifferenceStep5}. The proof is completed by gathering the above estimates.
\end{proof}

We are ready to prove Theorem \ref{Theo.WellposedLimitSys}. 
  
\begin{proof}[Proof of Theorem \ref{Theo.WellposedLimitSys}]  Suppose that $(\widetilde{R}, \widetilde{S}) \in L^{4+}(\Omega_T)\times W^{(2,1)}_{4+}(\Omega_T)$ is a weak solution to   (\ref{System.LimitingSys}) with respect to the initial data $(\widetilde R_0,\widetilde S_0)$. We will prove that it will coincide with the solution $(R,S)$. We first recall that $\widetilde S, S$ and $|\nabla \widetilde S|, |\nabla S|$ are bounded in $L^\infty(\Omega_T)$ due to the heat regularisation.  By direct computations,  
\begin{align*}
\big(h(\widetilde{R},\widetilde{S}) - h(R,S)\big)(\widetilde{R}-R) & = -   (\gamma_1 \xi_1(\widetilde R,\widetilde S) +\gamma_2 \xi_2(\widetilde R,\widetilde S) )(\widetilde{R}-R)^2 + J(x,t)
\end{align*}
in which 
\begin{align*}
 J(x,t) := \,&  (\gamma_1\Rhat-\eta_1-\gamma_1 R)(\xi_1(\widetilde R,\widetilde S) -  \xi_1(R,S))(\widetilde{R}-R)\\
 + \,&  (\gamma_2\Rhat-\eta_2-\gamma_2 R)(\xi_2(\widetilde R,\widetilde S) -  \xi_2(R,S))(\widetilde{R}-R).  
\end{align*}
Now, by taking the test function $\varphi= \widetilde R -R$ in \eqref{System.WeakSolution} 
\begin{equation}
\begin{aligned}
& \frac{1}{2}  \dintt{}  \partial_t(\widetilde{R}-R)^2  + \dintt{} \nabla \bigg( \bigg( d_{R_1} - (d_{R_1}-d_{R_2})\frac{f(\widetilde S)}{f(\widetilde S)+g(\widetilde S)} \bigg) (\widetilde{R}-R) \bigg) \cdot \nabla (\widetilde{R}-R) \\
& \le (d_{R_1}-d_{R_2}) \dintt{} \nabla \bigg( \bigg( \frac{f(\widetilde S)}{f(\widetilde S)+g(\widetilde S)} - \frac{f(S)}{f(S)+g(S)} \bigg) R \bigg) \cdot \nabla (\widetilde{R}-R) + \dintt{} J(x,\tau) . 
\end{aligned}
\label{Proof.TheoUniqueStep1}
\end{equation}
For the term $J$, since $R\in L^\infty(\Omega_T)$ as Lemma \ref{Lem.LInfinityEstimate}, 
it is clear that 
\begin{align*}
\dintt{} J
& \le \frac{1}{2} \left( \|\gamma_1\Rhat-\eta_1-\gamma_1 R \|_{L^\infty(\Omega_T)}^2 \dintt{} | \xi_1(\widetilde R,\widetilde S) - \xi_1(R,S) |^2 +  \dintt{} |\widetilde{R}-R |^2 \right) \\ 
& + \frac{1}{2} \left( \|\gamma_2\Rhat-\eta_2-\gamma_2 R \|_{L^\infty(\Omega_T)}^2 \dintt{} | \xi_2(\widetilde R, \widetilde S) - \xi_2(R,S) |^2 +  \dintt{} |\widetilde{R}-R |^2 \right) ,
\end{align*} 
which, thanks to \eqref{Proof.LInfDifferenceStep5}, yields 
\begin{align}
\dintt{} J \le C_T\dintt{} \left( |\widetilde R-R |^2 + |\widetilde S-S |^2 \right). 
\label{Proof.TheoUniqueStep2}
\end{align}

We proceed to treat the terms including gradients. Let us begin with the term on the left-hand side of \eqref{Proof.TheoUniqueStep1}. A similar estimate as  \eqref{Proof.LemLInfinityStep2} shows
\begin{equation}
\begin{aligned}
& \dintt{} \nabla \bigg( \bigg( d_{R_1} - (d_{R_1}-d_{R_2})\frac{f(\widetilde S)}{f(\widetilde S)+g(\widetilde S)} \bigg) (\widetilde{R}-R) \bigg) \cdot \nabla (\widetilde{R}-R)\\
& = \dintt{} \bigg( d_{R_1} - (d_{R_1}-d_{R_2})\frac{f(\widetilde S)}{f(\widetilde S)+g(\widetilde S)}  \bigg) |\nabla (\widetilde{R}-R)|^2 \\
& - (d_{R_1}-d_{R_2}) \dintt{} (\widetilde{R}-R) \nabla \bigg( \frac{f(\widetilde S)}{f(\widetilde S)+g(\widetilde S)}  \bigg) \cdot \nabla (\widetilde{R}-R) \\
& \ge d_{R_2}  \dintt{} |\nabla (\widetilde{R}-R)|^2
- C_{T}\|\nabla S\|_{L^\infty(\Omega_T)^2} \int_\Omega |\widetilde{R}-R| |\nabla (\widetilde{R}-R)| \\
& \ge   \frac{4d_{R_2}}{5}  \dintt{} |\nabla (\widetilde{R}-R)|^2
-  \frac{5C_{T}^2\|\nabla S\|_{L^\infty(\Omega_T)^2}^2}{4} \dintt{} |\widetilde{R}-R|^2 , 
\end{aligned}
\label{Proof.TheoUniqueStep3}
\end{equation}
where noting that $d_{R_1} - (d_{R_1}-d_{R_2}) f(\widetilde S)/(f(\widetilde S)+g(\widetilde S)) \ge d_{R_2}.$ 
Next, let us consider the term with gradients on the right hand side of \eqref{Proof.TheoUniqueStep1}, which can be treated by applying Lemma \ref{Lem.DifficultTerm}. 
For this term, the computation   
\begin{align*}
& \nabla \bigg( \frac{f(\widetilde S)}{f(\widetilde S)+g(\widetilde S)} - \frac{f(S)}{f(S)+g(S)} \bigg) \\
& =  \frac{1}{(f(\widetilde S)+g(\widetilde S))^2} \Big( (f'(\widetilde S)-f'(S))g(\widetilde S) \nabla \widetilde S +  f'(S)(g(\widetilde S)-g(S)) \nabla \widetilde S + f'(S)g(S) \nabla (\widetilde S - S) \Big) \\
& - \frac{1}{(f(\widetilde S)+g(\widetilde S))^2} \Big( (f(\widetilde S)-f(S))g'(\widetilde S) \nabla \widetilde S +  f(S)(g'(\widetilde S)-g'(S)) \nabla \widetilde S + f(S)g'(S) \nabla (\widetilde S - S) \Big) \\
& +  \frac{f(\widetilde S)+g(\widetilde S)+f(S)+g(S) }{(f(\widetilde S)+g(\widetilde S))^2(f(S)+g(S))^2} \nabla S \Big(f'(S)g(S) - f(S)g'(S)\Big) \Big( f(\widetilde S)-f(S)+ g(\widetilde S)-g(S)  \Big),  
\end{align*}
the continuity of $f,f',g,g'$ on $\R_+$ (Assumption \ref{Assumption.AllinOne}), and the positively lower boundedness of $f(S)+g(S)$, $f(\widetilde S)+g(\widetilde S)$ (similarly as \eqref{Proof.LemLInfinityStep2}) yield that   
\begin{align}
\left|\nabla \bigg( \frac{f(\widetilde S)}{f(\widetilde S)+g(\widetilde S)} - \frac{f(S)}{f(S)+g(S)} \bigg)\right| \le C_{T,\|(S,\nabla S,\widetilde S,\nabla \widetilde S)\|_{L^\infty(\Omega_T)^6}} |\widetilde S - S|,
\label{Proof.TheoUniqueStep4}
\end{align}
where we recall the boundedness of $\nabla S, \nabla \widetilde S$ from Lemma \ref{Lem.LInfinityEstimate} and the estimate \eqref{Proof.LInfDifferenceStep2}. This can be plugged into the following estimate 
\begin{align*}
& \dintt{} R \nabla \bigg( \frac{f(\widetilde S)}{f(\widetilde S)+g(\widetilde S)} - \frac{f(S)}{f(S)+g(S)} \bigg) \cdot \nabla (\widetilde{R}-R) \\
& \le C_{T}\|R\|_{L^\infty(\Omega_T)}^2  \dintt{} \bigg| \nabla \bigg( \frac{f(\widetilde S)}{f(\widetilde S)+g(\widetilde S)} - \frac{f(S)}{f(S)+g(S)} \bigg)\bigg|^2 + \frac{d_{R_2}}{5} \dintt{} |\nabla (\widetilde{R}-R)|^2  \\
&\le C_{T,\|(S,\nabla S,\widetilde S,\nabla \widetilde S)\|_{L^\infty(\Omega_T)^6}} \|R\|_{L^\infty(\Omega_T)}^2   
\left( \dintt{} |\widetilde S - S|^2 \right) + \frac{d_{R_2}}{5} \dintt{} |\nabla (\widetilde{R}-R)|^2.
\end{align*}
Here, we recall that $R$ is bounded in $L^\infty(\Omega_T)$ from Lemma \ref{Lem.LInfinityEstimate}.
Thanks to Lemma \ref{Lem.DifficultTerm},  
\begin{equation} 
\begin{aligned}
& \dintt{} \nabla \bigg( \bigg( \frac{f(\widetilde S)}{f(\widetilde S)+g(\widetilde S)} - \frac{f(S)}{f(S)+g(S)} \bigg) R \bigg) \cdot \nabla (\widetilde{R}-R) \\
& = \dintt{} \bigg( \frac{f(\widetilde S)}{f(\widetilde S)+g(\widetilde S)} - \frac{f(S)}{f(S)+g(S)} \bigg) \nabla R \cdot \nabla (\widetilde{R}-R) \\
& + \dintt{} R \nabla \bigg( \frac{f(\widetilde S)}{f(\widetilde S)+g(\widetilde S)} - \frac{f(S)}{f(S)+g(S)} \bigg) \cdot \nabla (\widetilde{R}-R) \\
&\le C_{T} \dintt{} \left( |\widetilde R-R |^2 + |\widetilde S-S |^2 \right) + \frac{2d_{R_2}}{5} \dintt{}  |\nabla (\widetilde{R}-R)|^2 \\
& + C_T \|\widetilde S_0 - S_0\|_{W^{2,1}_{2+}(\Omega)}^2 . 
\end{aligned}
\label{Proof.TheoUniqueStep5}
\end{equation}
Now, we also note from the proof of Lemma \ref{Lem.DifficultTerm}, specifically the estimate \eqref{Proof.LemDifficultTerm},   that 
\begin{align*}
& \int_\Omega |\widetilde S-S|^2 +  d_S\dintt{} |\nabla (\widetilde S - S)|^2 \\
&
 \le C \|\widetilde S_0 - S_0\|_{W^{2,1}_{2+}(\Omega)}^2 + C_{T} \dintt{} \left( |\widetilde R-R |^2 + |\widetilde S-S |^2 \right) .  
\end{align*}

Taking all the estimates above, we arrive at  
\begin{align*}
& \int_\Omega \left( |\widetilde R - R|^2 + |\widetilde S - S|^2 \right) + \frac{d_{R_2}}{5} \dintt{} |\nabla(\widetilde R-R) |^2 +   \dintt{}  |\nabla(\widetilde S-S) |^2  \\
& \le C \|\widetilde R_0 - R_0\|_{L^2(\Omega)}^2 + C\|\widetilde S_0 - S_0\|_{W^{2,1}_{2+}(\Omega)}^2  +  C_T\dintt{} \left( |\widetilde R-R |^2 + |\widetilde S-S |^2 \right), 
\end{align*}
which after applying the Gr\"onwall inequality yields $(\widetilde R, \widetilde S) \equiv (R,S)$ if they have the same initial data, i.e. the uniqueness is proved. Moreover, we also have the following stability estimate for the first component of the solution
\begin{align}
 \|\widetilde R - R\|_{L^2(0,T;H^1(\Omega))} + \|\widetilde S -S \|_{L^{2}(\Omega_T)}  \le  C_T \|\widetilde R_0-R_0\|_{L^{2}(\Omega)} + C_T \|\widetilde S_0-S_0\|_{W^{2,1+}_{2+}(\Omega)}. 
\label{Proof.TheoUniqueStep6}
\end{align}
To have the stability for the second component of the solution, we note that  
\begin{align}
  \|\widetilde S-S\|_{W^{2,1}_{2}(\Omega_T)}  \le C_T \left( \|\widetilde R -R \|_{L^{2}(\Omega_T)} + \|\widetilde S -S \|_{L^{2}(\Omega_T)} \right) + C_T \|\widetilde S_0-S_0\|_{W^{2,1+}_{2+}(\Omega)} 
 \label{Proof.TheoUniqueStep7}
\end{align} 
due to an application of Lemma \ref{Lem.HeatRegularisation} to the equation \eqref{Equation.S} and using the estimate \eqref{Proof.LInfDifferenceStep5}. The stability \eqref{Estimate.Stability} is obtained by combining \eqref{Proof.TheoUniqueStep6} and \eqref{Proof.TheoUniqueStep7}. With $R,S\in L^\infty(\Omega_T)$, the claim $S\in W^{2,1}_q(\Omega_T)$, for any  $1\le q<\infty$, is directly obtained by applying Lemma \ref{Lem.HeatRegularisation}.  

\medskip

Next, we will prove that $(R,S)$ is also the strong solution to \eqref{System.LimitingSys}-\eqref{Condition.LimInitialCond}.  Since the weak solution is smooth enough, especially the component $S$, it suggests to rewrite the equation for $R$ as follows
\begin{align*}
\partial_t R - a(S)\Delta R = b(S)\cdot \nabla R + c(S)R + h(R,S) ,
\end{align*}
where  
\begin{align*}
a(S):= d_{R_1}-(d_{R_1}-d_{R_2}) \frac{f(S)}{f(S)+g(S)},
\end{align*}
and 
\begin{align*}
b_i(S):= 2 \frac{\partial a}{\partial S} \nabla S, \quad c(S):= \frac{\partial a}{\partial S} \Delta S + \frac{\partial^2 a}{\partial S^2} |\nabla S|^2.
\end{align*}
We recall from the proof of Theorem \ref{Theo.ConvergeToLimitSys} that $\nabla S\in L^\infty(\Omega_T) \cap W^{2,1}_q(\Omega_T)$ for any $1\le q<\infty$. Then,  $b(S)\in L^\infty(\Omega_T)^2$ and $c(S)\in L^q(\Omega_T)$. Moreover, by the Sobolev embedding, the function $S$ is bounded and uniformly continuous, and so is $a(S)$. Let $2\le q<\infty$. By the Aubin-Lions inequality, for any $\vartheta>0$ there exists $C_{\vartheta}>0$ such that 
\begin{align*}
\|\nabla R\|_{L^{q}(\Omega)} \le \vartheta \|\Delta R\|_{L^{2}(\Omega)} + C_{\vartheta} \|R\|_{L^{\infty}(\Omega)}.
\end{align*}
This subsequently shows
\begin{align*}
\|\nabla R\|_{L^{q}(\Omega_T)} \le \vartheta \|\Delta R\|_{L^q(\Omega_T)} + C_{\vartheta} \|R\|_{L^{\infty}(\Omega_T)}, 
\end{align*}
and therefore,
\begin{align*}
\|b(S)\cdot \nabla R + c(S)R + h(R,S)\|_{L^{q}(\Omega_T)} \le  C_{T,\vartheta} +  \vartheta C_T \|\Delta R\|_{L^q(\Omega_T)} . 
\end{align*}
Note that the operator $R \mapsto -a(S)\Delta R$ is uniformly elliptic since $a(S)\ge d_{R_2}>0$. Hence,  by the  maximal regularity for parabolic equations, see e.g.  \cite[Remark 51.5]{quittner2019superlinear}, 
\begin{align*}
& \|\partial_t R\|_{L^{q}(\Omega_T)} + \|\Delta R\|_{L^{q}(\Omega_T)}  \\
& \le C \left( \|R_{10}+R_{20}\|_{W^{2,q}(\Omega)} + \|b(S)\cdot \nabla R + c(S)R + h(R,S)\|_{L^{q}(\Omega_T)} \right) \\
& \le C \left( \|R_{10}+R_{20}\|_{W^{2,q}(\Omega)} + C_{T,\vartheta} +  \vartheta C_T \|\Delta R\|_{L^q(\Omega_T)} \right).
\end{align*}
With a sufficiently small choice of $\vartheta$, we get 
 \begin{align*}
 \|\partial_t R\|_{L^{q}(\Omega_T)} + \|\Delta R\|_{L^{q}(\Omega_T)}  \le C_{T,\vartheta}   \|R_{10}+R_{20}\|_{W^{2,q}(\Omega)} + C_{T,\vartheta}, 
\end{align*}
 and complete the proof. 
\end{proof}

\section{Convergence rate}\label{sec4}

\subsection{Projection onto a lower regularity space}
 
Let $(R^\varepsilon,S^\varepsilon)$ be the $\varepsilon$-dependent solution to the system \eqref{System.OrginalSys}-\eqref{Condition.Initial}, and  $(R,S)$ be the unique strong solution the limiting system \eqref{System.LimitingSys}  obtained by Theorem \ref{Theo.WellposedLimitSys}.  This section is devoted to study the convergence rate  $(|R^\varepsilon - R|,|S^\varepsilon-S|)$ in $L^2(\Omega_T)$, where the more challenging is the first component.   

\medskip

The main idea here is to project the equation for $R^\varepsilon - R$ onto a lower regularity space. More precisely, we fix a constant $\zeta>0$ and denote $-\Delta_\zeta:=-(\Delta-\zeta I)$. Then, the   Neumann boundary problem for the Poisson equation  
\begin{align}
\Delta_\zeta u = \phi \quad \text{in }  \Omega, \qquad
\nabla u \cdot \nu = 0 \quad \text{on }  \partial\Omega, 
\nonumber
\end{align}
with a given function $\phi \in L^2(\Omega)$, has a unique solution $u=\Delta_\zeta^{-1} \phi \in H^2(\Omega)$, see e.g. H. Brezis \cite[Chapter 9]{brezis2011functional}.  The following lemma is obvious since we have, thanks to the previous sections, $R,R^\varepsilon \in L^2(0,T;H^2(\Omega))$ as well as $h(R,S), h(R^\varepsilon,S^\varepsilon)\in L^2(\Omega_T)$.

\begin{lemma} \label{Lem.UXY} For each $\varepsilon>0$ and $t\in [0,T]$, the Neumann boundary problems for Poisson equations of finding $U^\varepsilon$, $X^\varepsilon$, $Y^\varepsilon$ such that 
\begin{align}
\Delta_\zeta U^\varepsilon = R^\varepsilon - R \hspace*{4cm} \text{in }  \Omega, \qquad
\nabla U^\varepsilon \cdot \nu = 0 \quad \text{on }  \partial\Omega, 
\label{Problem.U}\\
\Delta_\zeta X^\varepsilon = d_{R_1}(R_1^\varepsilon - \xi_1) + d_{R_2}(R_2^\varepsilon - \xi_2) \quad \text{in }  \Omega, \qquad
\nabla X^\varepsilon \cdot \nu = 0 \quad \text{on }  \partial\Omega, 
\label{Problem.X} \\
 \Delta_\zeta Y^\varepsilon = h(R^\varepsilon,S^\varepsilon) - h(R,S) \hspace*{1.95cm} \text{in }  \Omega, \qquad
\nabla Y^\varepsilon \cdot \nu = 0 \quad \text{on }  \partial\Omega, 
\label{Problem.Y}
\end{align}
have their uniquely strong solutions $U^\varepsilon(t,\cdot)$, $X^\varepsilon(t,\cdot)$, and $Y^\varepsilon(t,\cdot)$ such that 
\begin{gather*}
U^\varepsilon = \Delta_\zeta^{-1}(R^\varepsilon-R) \in L^2(0,T;H^4(\Omega)), \\
X^\varepsilon = \Delta_\zeta^{-1} \Big(d_{R_1}(R_1^\varepsilon - \xi_1) + d_{R_2}(R_2^\varepsilon - \xi_2)\Big) \in L^2(0,T;H^2(\Omega)), \\ 
Y^\varepsilon = \Delta_\zeta^{-1} (h(R^\varepsilon,S^\varepsilon) - h(R,S)) \in L^2(0,T;H^2(\Omega)) .
\end{gather*}
\end{lemma}

In order to get the equation for $R^\varepsilon - R$, we can add the equations for $R_1^\varepsilon$, $R_2^\varepsilon$, and then subtract the result equation by the equation for $R$, which arrive at 
\begin{align*}
 \partial_t (R^\varepsilon-R) =   \Delta \left( \sum_{i=1}^2 d_{R_i}(R_i^\varepsilon - \xi_i) \right) +  (h(R^\varepsilon,S^\varepsilon)- h(R,S)).  
\end{align*}
Then, with notations in Lemma \ref{Lem.UXY}, 
\begin{align}
 \partial_t U^\varepsilon =  (I+\zeta \Delta_\zeta^{-1})\Delta_\zeta X^\varepsilon + Y^\varepsilon. \label{Equation.Projection} 
\end{align} 
Instead of estimating the rate $\|R^\varepsilon-R\|_{L^2(\Omega_T)}$, we will estimate $\|\Delta_\zeta U^\varepsilon \|_{L^2(\Omega_T)}$. This will be based on the projected equation \eqref{Equation.Projection} upon multiplying two sides by $-\Delta_\zeta U^\varepsilon$. In the next section, we will estimate terms that appear according to this multiplication and then obtain the convergence rate. 

\subsection{Convergence rate of the fast reaction limit}

In this part, for short, we denote 
\begin{align*}
\begin{array}{cccccc}
& f^\varepsilon := f(S^\varepsilon), & g^\varepsilon := g(S^\varepsilon), & f^0 := f(S), & g^0:= g(S),\\
%%%%%%%%%%%%%%%%%%%%
&(f^\varepsilon)' := f'(S^\varepsilon), & (g^\varepsilon)' := g'(S^\varepsilon), & (f^0)' := f'(S), & (g^0):= g'(S),
\end{array}
\end{align*}
and prefer to write $\xi_1:=\xi_1(R,S)$, $\xi_2:=\xi_2(R,S)$ instead of $R_1,R_2$ respectively. 

\begin{lemma}
\label{Lem.ConvRateMixterm} Let $U^\varepsilon, X^\varepsilon$ be defined by Lemma \ref{Lem.UXY}. Then, for all $t\in(0,T)$, 
\begin{align*}
- \dintt{}  \Delta_\zeta U^\varepsilon (I+\zeta \Delta_\zeta^{-1}) \Delta_\zeta X^\varepsilon & \le   \frac{(d_{R_1}-d_{R_2})^2}{d_{R_2}} C_T \dintt{} \left( |f^\varepsilon R_1^\varepsilon   - g^\varepsilon R_2^\varepsilon|^2  + |S^\varepsilon - S |^2 \right) 
\\ 
& - \frac{d_{R_2}}{2} \dintt{}  |\Delta_\zeta U^\varepsilon|^2  + \frac{\zeta(\zeta+4)}{4} d_{R_2} \dintt{}     |U^\varepsilon|^2. 
\end{align*}
\end{lemma}
\begin{proof} By the notations in Lemma \ref{Lem.UXY} and direct computations, 
\begin{equation}
\begin{aligned}
\Delta_\zeta X^\varepsilon      
& =   (d_{R_1}-d_{R_2}) \left( \frac{f^\varepsilon R_1^\varepsilon   - g^\varepsilon R_2^\varepsilon}{f^\varepsilon + g^\varepsilon}  + \left( \frac{f^0}{f^0+g^0} - \frac{f^\varepsilon}{f^\varepsilon+g^\varepsilon} \right)\right)  \\
& + \left( d_{R_1}- (d_{R_1}-d_{R_2})\frac{f^\varepsilon}{f^\varepsilon+g^\varepsilon} \right)  \Delta_\zeta U^\varepsilon.  
\end{aligned}
\label{Proof.LemConvRateMixterm1}
\end{equation}
In addition, we can use  integration by parts to pass the operator $\Delta_\zeta$ from $U^\varepsilon$ to $X^\varepsilon$ as  
\begin{align*}
\dintt{}  \Delta_\zeta U^\varepsilon   X^\varepsilon = \dintt{}   U^\varepsilon  \Delta_\zeta X^\varepsilon. 
\end{align*} 
These together yield that 
\begin{equation}
\begin{aligned}
& - \dintt{}  \Delta_\zeta U^\varepsilon (I+\zeta \Delta_\zeta^{-1}) \Delta_\zeta X^\varepsilon \\
& = - (d_{R_1}-d_{R_2}) \dintt{} \left( \frac{f^\varepsilon R_1^\varepsilon   - g^\varepsilon R_2^\varepsilon}{f^\varepsilon + g^\varepsilon}  + \left( \frac{f^0}{f^0+g^0} - \frac{f^\varepsilon}{f^\varepsilon+g^\varepsilon} \right)\right) \Delta_\zeta U^\varepsilon  \\
& - \dintt{} \left( d_{R_1}- (d_{R_1}-d_{R_2})\frac{f^\varepsilon}{f^\varepsilon+g^\varepsilon} \right) |\Delta_\zeta U^\varepsilon|^2 - \zeta \dintt{}  \Delta_\zeta U^\varepsilon   X^\varepsilon \\
& \le  - (d_{R_1}-d_{R_2}) \dintt{} \left( \frac{f^\varepsilon R_1^\varepsilon   - g^\varepsilon R_2^\varepsilon}{f^\varepsilon + g^\varepsilon}  + \left( \frac{f^0}{f^0+g^0} - \frac{f^\varepsilon}{f^\varepsilon+g^\varepsilon} \right)\right) \Delta_\zeta U^\varepsilon  \\
& \hspace*{4.22cm} - d_{R_2} \dintt{}  |\Delta_\zeta U^\varepsilon|^2 - \zeta \dintt{}   U^\varepsilon \Delta_\zeta  X^\varepsilon,  
\end{aligned}
\label{Proof.LemConvRateMixterm2}
\end{equation}
where we have bounded $ d_{R_1}- (d_{R_1}-d_{R_2})f^\varepsilon/(f^\varepsilon+g^\varepsilon)$ from below by $d_{R_2}$. 
Let us consider the latter term in  \eqref{Proof.LemConvRateMixterm2}. 
By the computation \eqref{Proof.LemConvRateMixterm1} again, we have 
\begin{equation}
\begin{aligned}
-\zeta \dintt{}   U^\varepsilon \Delta_\zeta  X^\varepsilon &= -\zeta (d_{R_1}-d_{R_2}) \dintt{}    \left( \frac{f^\varepsilon R_1^\varepsilon   - g^\varepsilon R_2^\varepsilon}{f^\varepsilon + g^\varepsilon}  + \left( \frac{f^0}{f^0+g^0} - \frac{f^\varepsilon}{f^\varepsilon+g^\varepsilon} \right)\right) U^\varepsilon  \\
& - \zeta \dintt{}    \left( d_{R_1}- (d_{R_1}-d_{R_2})\frac{f^\varepsilon}{f^\varepsilon+g^\varepsilon} \right) U^\varepsilon \Delta_\zeta U^\varepsilon \\
%%%%%%%%
&\le -\zeta (d_{R_1}-d_{R_2}) \dintt{}    \left( \frac{f^\varepsilon R_1^\varepsilon   - g^\varepsilon R_2^\varepsilon}{f^\varepsilon + g^\varepsilon}  + \left( \frac{f^0}{f^0+g^0} - \frac{f^\varepsilon}{f^\varepsilon+g^\varepsilon} \right)\right) U^\varepsilon  \\
& + \zeta d_{R_2} \dintt{}     |U^\varepsilon|^2 + \frac{d_{R_2}}{4} \dintt{} |\Delta_\zeta U^\varepsilon|^2. 
\end{aligned}
\label{Proof.LemConvRateMixterm3}
\end{equation}
Now, we recall that $|f^0/(f^0+g^0) - f^\varepsilon/(f^\varepsilon+g^\varepsilon)|$ can be bounded by $C_T|S^\varepsilon-S|$ due to a similar argument as \eqref{Proof.LInfDifferenceStep1}-\eqref{Proof.LInfDifferenceStep2}. 
  Therefore, \eqref{Proof.LemConvRateMixterm2}-\eqref{Proof.LemConvRateMixterm3} implies that 
\begin{equation}
\begin{aligned}
& - \dintt{}  \Delta_\zeta U^\varepsilon (I+\zeta \Delta_\zeta^{-1}) \Delta_\zeta X^\varepsilon \\
& \le  - (d_{R_1}-d_{R_2}) \dintt{} \left( \frac{f^\varepsilon R_1^\varepsilon   - g^\varepsilon R_2^\varepsilon}{f^\varepsilon + g^\varepsilon}  + \left( \frac{f^0}{f^0+g^0} - \frac{f^\varepsilon}{f^\varepsilon+g^\varepsilon} \right)\right)  \Delta U^\varepsilon  \\
& - \frac{3d_{R_2}}{4} \dintt{}  |\Delta_\zeta U^\varepsilon|^2  + \zeta d_{R_2} \dintt{}     |U^\varepsilon|^2  \\
%%%%%%%%%%%%%%%%%%%%% 
& \le  \frac{2(d_{R_1}-d_{R_2})^2}{d_{R_2}}  \dintt{}  \left( \frac{f^\varepsilon R_1^\varepsilon   - g^\varepsilon R_2^\varepsilon}{f^\varepsilon + g^\varepsilon}  + \left( \frac{f^0}{f^0+g^0} - \frac{f^\varepsilon}{f^\varepsilon+g^\varepsilon} \right)\right)^2 
\\ 
& + \frac{d_{R_2}}{8} \dintt{}  |\Delta U^\varepsilon|^2 - \frac{3d_{R_2}}{4} \dintt{}  |\Delta_\zeta U^\varepsilon|^2  + \zeta d_{R_2} \dintt{}     |U^\varepsilon|^2  \\
%%%%%%%%%%%%%%%%%%%%% 
& \le  \frac{(d_{R_1}-d_{R_2})^2}{d_{R_2}} C_T \dintt{} \left( |f^\varepsilon R_1^\varepsilon   - g^\varepsilon R_2^\varepsilon|^2  + |S^\varepsilon - S^0|^2 \right) 
\\ 
& \hspace*{2.2cm} - \frac{d_{R_2}}{2} \dintt{}  |\Delta_\zeta U^\varepsilon|^2  + \frac{\zeta(\zeta+4)}{4} d_{R_2} \dintt{} |U^\varepsilon|^2 ,
\end{aligned}
\end{equation}
where we note $|\Delta U^\varepsilon|^2 \le 2 |\Delta_\zeta U^\varepsilon|^2 + 2\zeta^2 |U^\varepsilon|^2$.  
\end{proof}

In the next lemma, based on estimating the fast reaction term in Lemma \ref{Lem.Bootstrap}, we will see that the component distances $|R_1^\varepsilon-R_1|$, $|R_2^\varepsilon-R_2|$ can be estimated whenever $|R^\varepsilon - R|$, $|S^\varepsilon-S|$ can be.

\begin{lemma} 
\label{Lem.AbsoluteDistance}
There exists $C_T>0$ so that
\begin{align}
|R_1^\varepsilon-\xi_1|  + |R_2^\varepsilon-\xi_2| \le C_T |f^\varepsilon R_1^\varepsilon   - g^\varepsilon R_2^\varepsilon| + C_T \Big( |R^\varepsilon - R| + |S^\varepsilon-S| \Big)   
\label{Estimate.AbsoluteDistance}
\end{align}
pointwise on $\Omega_T$.
\end{lemma}

\begin{proof} Let $K_1^\varepsilon,K_2^\varepsilon$ be defined by \eqref{Proof.TheoFSL.Step4}. Since $R\in L^\infty(\Omega_T)$ (Lemma \ref{Lem.LInfinityEstimate}), we have
\begin{align*}
  |R_1^\varepsilon - \xi_1| + |R_2^\varepsilon - \xi_2|   
& \le  \left|R_1^\varepsilon - \frac{g^\varepsilon}{f^\varepsilon+g^\varepsilon} R^\varepsilon \right| +  \frac{g^\varepsilon}{f^\varepsilon+g^\varepsilon} |R^\varepsilon-R| + \left| K_2^\varepsilon R \right|  \\
%%%%%%%%%%%%%%%%%%%%%%%%%%
& + \left|R_2^\varepsilon - \frac{f^\varepsilon}{f^\varepsilon+g^\varepsilon} R^\varepsilon \right| +  \frac{f^\varepsilon}{f^\varepsilon+g^\varepsilon} |R^\varepsilon-R| + \left| K_1^\varepsilon R \right|  \\
%%%%%%%%%%%%%%%%%%%%%%%%%%
& \le 2   \bigg( \sup_{\varepsilon>0} \sup_{\Omega_T} \frac{1}{f^\varepsilon +g^\varepsilon} \bigg) | f^\varepsilon R_1^\varepsilon   - g^\varepsilon R_2^\varepsilon  | +  |R^\varepsilon-R| \\
& + 2 \bigg( \sup_{\varepsilon>0} \sup_{\Omega_T} \frac{1}{f^\varepsilon +g^\varepsilon} \bigg) \Big( |f^\varepsilon-f^0| + |g^\varepsilon-g^0| \Big) \|R\|_{L^\infty(\Omega_T)}.
\end{align*} 
The desired estimate follows from the estimate \eqref{Proof.TheoFSL.Step3}, and the uniform lower boundedness of $f^\varepsilon, g^\varepsilon$ from the proof of Lemma \ref{Lem.Feedback}.   
\end{proof}

\begin{lemma}
\label{Lem.NonTerm} There exists $C_T>0$ so that
\begin{align*}
-\dintt{}  U^\varepsilon (h(R^\varepsilon,S^\varepsilon) - h(R,S))   & \le C_{T} \dintt{} \left( |U^\varepsilon|^2 + |\nabla U^\varepsilon|^2 \right) + \frac{d_{R_2}}{4} \dintt{} |\Delta_\zeta U^\varepsilon|^2 \\
&+ C_T  \dintt{}  \Big(  |f^\varepsilon R_1^\varepsilon   - g^\varepsilon R_2^\varepsilon|^2   +  |S^\varepsilon-S|^2 \Big) , 
\end{align*}
for all $t\in(0,T)$.
\end{lemma}

\begin{proof} By employing the notations in Lemma \ref{Lem.UXY}, we have 
\begin{align*}
& |h(R^\varepsilon,S^\varepsilon) - h(R,S)| \\
& = \Big| (R_1^\varepsilon - \xi_1) \Big( \gamma_1 \Rhat - \eta_1 - \gamma_1 (R_1^\varepsilon + R) \Big)  + (R_2^\varepsilon - \xi_2) \Big( \gamma_2 \Rhat - \eta_2 - \gamma_2 (R_2^\varepsilon + R) \Big) \Big| \\
& \le \Big( |\gamma_1 \Rhat - \eta_1|+|\gamma_2 \Rhat - \eta_2|  + (\gamma_1+\gamma_2) (R^\varepsilon + 2R) \Big) \Big( |R_1^\varepsilon - \xi_1| + |R_2^\varepsilon - \xi_2| \Big),
\end{align*}
which, thanks to Lemma \ref{Lem.AbsoluteDistance}, implies that  
\begin{align*}
|h(R^\varepsilon,S^\varepsilon) - h(R,S)|  
& \le C_{T} (1 + R^\varepsilon + R)  \Big( |f^\varepsilon R_1^\varepsilon   - g^\varepsilon R_2^\varepsilon| +  |\Delta_\zeta U^\varepsilon| + |S^\varepsilon-S| \Big) \\
& \le C_{T} (1 + R)  \Big( |f^\varepsilon R_1^\varepsilon   - g^\varepsilon R_2^\varepsilon| +  |\Delta_\zeta U^\varepsilon| + |S^\varepsilon-S| \Big)  \\
& + C_{T} |\Delta_\zeta U^\varepsilon|  \Big( |f^\varepsilon R_1^\varepsilon   - g^\varepsilon R_2^\varepsilon| +  |\Delta_\zeta U^\varepsilon| + |S^\varepsilon-S| \Big)    ,  
\end{align*}
where recalling $R^\varepsilon - R = \Delta_\zeta U^\varepsilon$. 
 Hence, for any $\beta_1>0$,
\begin{equation}
\begin{aligned}
& -\dintt{}  U^\varepsilon (h(R^\varepsilon,S^\varepsilon) - h(R,S))   \\
& \le C_T \dintt{} (1 + 3R) |U^\varepsilon | \Big( |f^\varepsilon R_1^\varepsilon   - g^\varepsilon R_2^\varepsilon| +  |\Delta_\zeta U^\varepsilon| + |S^\varepsilon-S| \Big)         \\
& + C_T \dintt{} |U^\varepsilon \Delta_\zeta U^\varepsilon|  \Big( |f^\varepsilon R_1^\varepsilon   - g^\varepsilon R_2^\varepsilon| +  |\Delta_\zeta U^\varepsilon| + |S^\varepsilon-S| \Big) \\
%%%%%%%%%%%%%%%%%%%%%%
& \le C_T  \dintt{}  \Big( C_{\beta_1} |U^\varepsilon |^2 +  |f^\varepsilon R_1^\varepsilon   - g^\varepsilon R_2^\varepsilon|^2 + \beta_1 |\Delta_\zeta U^\varepsilon|^2  +  |S^\varepsilon-S|^2 \Big)         \\
& + C_T \dintt{} |U^\varepsilon \Delta_\zeta U^\varepsilon|  \Big( |f^\varepsilon R_1^\varepsilon   - g^\varepsilon R_2^\varepsilon| +  |\Delta_\zeta U^\varepsilon| + |S^\varepsilon-S| \Big) 
\end{aligned}
\label{Proof.LemNonTerm1}
\end{equation}
where we have used $R \in L^\infty(\Omega_T)$. 

\medskip

Now, thanks to Lemma \ref{Lem.Feedback}, the energy $E_p^\varepsilon(t)$ is bounded for any $1\le p<\infty$. Since $f(S^\varepsilon)$ and $g(S^\varepsilon)$ have positively lower bounds, from the proof of Lemma \ref{Lem.Feedback}, the components $R_1^\varepsilon$, $R_2^\varepsilon$ are uniformly bounded in $L^\infty(0,T;L^p(\Omega))$ for any $1\le p<\infty$. Therefore, in particular with $p=2$, the term $ U^\varepsilon=\Delta_\zeta^{-1}(R^\varepsilon -R)$ is uniformly bounded in $L^\infty(0,T;H^2(\Omega))$.  We then have 
\begin{align*}
\dintt{} |U^\varepsilon \Delta_\zeta U^\varepsilon|^2 \le \int_0^t \|U^\varepsilon\|_{L^\infty(\Omega)}^2 \|\Delta_\zeta U^\varepsilon\|_{L^2(\Omega)}^2 \le \|\Delta_\zeta U^\varepsilon\|_{L^\infty(0,T;L^2(\Omega))}^2 \int_0^t \|U^\varepsilon\|_{L^\infty(\Omega)}^2. 
\end{align*}
For any $\beta_2>0$,  the Sobolev  and  Gagliardo-Nirenberg inequalities yield  
\begin{align*}
\|U^\varepsilon\|_{L^\infty(\Omega)}^2 & \le C \left( \|U^\varepsilon\|_{L^{4}(\Omega)}^2 + \|\nabla U^\varepsilon\|_{L^{4}(\Omega)}^2 \right) \\
& \le C_{\beta_2,\zeta} \|U^\varepsilon\|_{L^{2}(\Omega)}^2 + C_{\beta_2} \|\nabla U^\varepsilon\|_{L^{2}(\Omega)}^2 + \beta_2 \|\Delta_\zeta U^\varepsilon\|_{L^{2}(\Omega)}^2 ,
\end{align*}
which consequently shows 
\begin{align*}
\dintt{} |U^\varepsilon \Delta_\zeta U^\varepsilon|^2    \le C_{T,\beta_2,\zeta} \dintt{} \left( |U^\varepsilon|^2 + |\nabla U^\varepsilon|^2 \right) + \beta_2 C_T \dintt{} |\Delta_\zeta U^\varepsilon|^2 . 
\end{align*}
Taking into account this estimate, for any $\beta_3>0$ we arrive at 
\begin{equation}
\begin{aligned}
& \dintt{} |U^\varepsilon \Delta_\zeta U^\varepsilon|  \Big( |f^\varepsilon R_1^\varepsilon   - g^\varepsilon R_2^\varepsilon| +  |\Delta_\zeta U^\varepsilon| + |S^\varepsilon-S| \Big) \\
& \le C_{\beta_3} \dintt{} |U^\varepsilon \Delta_\zeta U^\varepsilon|^2 + \dintt{} \Big( |f^\varepsilon R_1^\varepsilon   - g^\varepsilon R_2^\varepsilon|^2 + \beta_3 |\Delta_\zeta U^\varepsilon| + |S^\varepsilon-S|^2 \Big) \\
& \le C_{T,\beta_2,\beta_3,\zeta} \dintt{} \left( |U^\varepsilon|^2 + |\nabla U^\varepsilon|^2 \right) + (\beta_2C_{\beta_3} + \beta_3) C_T \dintt{} |\Delta_\zeta U^\varepsilon|^2 . 
\end{aligned}
\label{Proof.LemNonTerm2}
\end{equation}
Finally, a combination of \eqref{Proof.LemNonTerm1}-\eqref{Proof.LemNonTerm2} deduces 
\begin{align*}
& \dintt{}  U^\varepsilon (h(R^\varepsilon,S^\varepsilon) - h(R,S))   \\
%%%%%%%%%%%%%%%%%%%%%%
& \le C_T  \dintt{}  \Big( C_{\beta_1} |U^\varepsilon |^2 +  |f^\varepsilon R_1^\varepsilon   - g^\varepsilon R_2^\varepsilon|^2 + \beta_1 |\Delta_\zeta U^\varepsilon|^2  +  |S^\varepsilon-S|^2 \Big)         \\
& + C_{T,\beta_2,\beta_3} \dintt{} \left( |U^\varepsilon|^2 + |\nabla U^\varepsilon|^2 \right) + (\beta_2C_{\beta_3} + \beta_3) C_T \dintt{} |\Delta_\zeta U^\varepsilon|^2 ,
\end{align*} 
which directly implies the desired estimate by choosing $\beta_1,\beta_3$ and then $\beta_2$ small enough. 
\end{proof}
 
We proceed to prove Theorem \ref{Theo.ConRate}.

\begin{proof}[Proof of Theorem \ref{Theo.ConRate}] To prove this theorem, we will combine Lemmas \ref{Lem.ConvRateMixterm}, \ref{Lem.NonTerm}, and then apply the Gr\"onwall lemma.  First,
multiplying two sides of \eqref{Equation.Projection} by $-\Delta_\zeta U^\varepsilon$ gives 
\begin{equation}
\begin{aligned}
& \frac{1}{2} \frac{d}{dt} \dintt{} (\zeta |U^\varepsilon|^2 + |\nabla U^\varepsilon|^2)  \\
& \hspace{2cm} = - \dintt{}  \Delta_\zeta U^\varepsilon (I+\zeta \Delta_\zeta^{-1}) \Delta_\zeta X^\varepsilon  - \dintt{}  U^\varepsilon \Delta_\zeta Y^\varepsilon  \\
& \hspace{2cm} = -   \dintt{}  \Delta_\zeta U^\varepsilon (I+\zeta \Delta_\zeta^{-1}) \Delta_\zeta X^\varepsilon   -    \dintt{}  U^\varepsilon (h(R^\varepsilon,S^\varepsilon) - h(R,S)) \\
&\hspace{2cm} \le C_T \dintt{} (\zeta |U^\varepsilon|^2 + |\nabla U^\varepsilon|^2) - \frac{d_{R_2}}{4} \dintt{} |\Delta_\zeta U^\varepsilon|^2 \\
& \hspace{2cm} + C_T  \dintt{}    |f^\varepsilon R_1^\varepsilon   - g^\varepsilon R_2^\varepsilon|^2   + C_T  \dintt{} |S^\varepsilon-S|^2 ,   
\end{aligned}
\label{Proof.TheoConRate1}
\end{equation}
where the last estimate comes from applications of Lemmas \ref{Lem.ConvRateMixterm} and \ref{Lem.NonTerm}. Let us estimate the difference $V^\varepsilon= S^\varepsilon-S$ in $L^2(\Omega_t)$, which is the solution to 
\begin{align*}
\pa_t V^\varepsilon - d_S\Delta V^\varepsilon = \displaystyle   \mu \eta_1(R_1^\varepsilon -\xi_1) +  \mu  \eta_2 (R_2^\varepsilon-\xi_2)   -\rho V^\varepsilon.
\end{align*}
Indeed, by the estimate for $|R_1^\varepsilon-\xi_1$, $|R_2^\varepsilon-\xi_2|$ given by Lemma \ref{Lem.AbsoluteDistance}, we have   
\begin{equation}
\begin{aligned}
& \frac{1}{2} \frac{d}{dt} \dintt{}  (S^\varepsilon-S)^2 + \rho \dintt{} |S^\varepsilon-S|^2 + d_S\dintt{} |\nabla (S^\varepsilon -S)|^2 \\
 &=  \dintt{} \Big( \mu \eta_1|R_1^\varepsilon -\xi_1|  +  \mu  \eta_2 |R_2^\varepsilon-\xi_2|   \Big)|S^\varepsilon-S|    \\
&\le C_T \dintt{}  \Big( |f^\varepsilon R_1^\varepsilon   - g^\varepsilon R_2^\varepsilon| +  |\Delta_\zeta U^\varepsilon| + |S^\varepsilon-S| \Big) |S^\varepsilon-S| \\
&\le C_T \dintt{}  \Big( |f^\varepsilon R_1^\varepsilon   - g^\varepsilon R_2^\varepsilon|^2 + \beta_4 |\Delta_\zeta U^\varepsilon|^2   + C_{\beta_4} |S^\varepsilon-S|^2 \Big),   
\end{aligned}
\label{Proof.TheoConRate2}
\end{equation}
for any $\beta_4>0$. Note that $U^\varepsilon(0)=0$ since $R^\varepsilon$, $R$ have the same initial data, which also says that $\nabla U^\varepsilon(0)=0$. Hence, a combination of  \eqref{Proof.TheoConRate1} and \eqref{Proof.TheoConRate2} accordingly gives 
\begin{align*}
& \int_\Omega ( |U^\varepsilon|^2 + |\nabla U^\varepsilon|^2 ) + \int_\Omega (S^\varepsilon-S)^2  +   \dintt{} |\Delta_\zeta U^\varepsilon|^2 + \dintt{} |\nabla (S^\varepsilon -S)|^2 \\
& \le C_T \left( \dintt{} ( |U^\varepsilon|^2 + |\nabla U^\varepsilon|^2)  + \dintt{} |S^\varepsilon-S|^2  +   \dintt{}    |f^\varepsilon R_1^\varepsilon   - g^\varepsilon R_2^\varepsilon|^2 \right),    
\end{align*}  
where we have chosen $\beta_4$ small enough.  Hence, by the Gr\"onwall inequality,
\begin{align*}
& \int_\Omega ( |U^\varepsilon|^2  + \int_\Omega (S^\varepsilon-S)^2  + |\nabla U^\varepsilon|^2 ) \le   C_T  \dintt{}    |f^\varepsilon R_1^\varepsilon   - g^\varepsilon R_2^\varepsilon|^2.   
\end{align*}
Consequently, 
\begin{equation}
\begin{aligned}
& \dintt{} |\Delta_\zeta U^\varepsilon|^2  + \int_\Omega (S^\varepsilon-S)^2   \le   C_T  \dintt{}    |f^\varepsilon R_1^\varepsilon   - g^\varepsilon R_2^\varepsilon|^2,    
\end{aligned}
\label{Proof.TheoConRate3}
\end{equation}
which shows that  
\begin{align*}
& \|R^\varepsilon-R\|_{L^2(\Omega_T)}  + \|S^\varepsilon-S\|_{L^\infty(0,T;L^2(\Omega))}  \le   C_T \|f^\varepsilon R_1^\varepsilon   - g^\varepsilon R_2^\varepsilon\|_{L^2(\Omega_T)} ,   
\end{align*} 
where we recast $\Delta_\zeta U^\varepsilon$ by $R^\varepsilon-R$. Eventually, by the heat regularisation (as given in Lemma \ref{Lem.HeatRegularisation})   applied to the equation for $V^\varepsilon$, 
\begin{align*}
\|S^\varepsilon-S\|_{W^{2,1}_2(\Omega_T)} 
& \le C_T \Big(  \|R_1^\varepsilon -\xi_1\|_{L^2(\Omega)} + \|R_2^\varepsilon-\xi_2 \|_{L^2(\Omega)} \Big) \le C_T \|f^\varepsilon R_1^\varepsilon   - g^\varepsilon R_2^\varepsilon\|_{L^2(\Omega_T)}, 
\end{align*}
where we combined Lemma \ref{Lem.AbsoluteDistance}  and the estimate \eqref{Proof.TheoConRate3}  at the second estimate. The theorem is proved by taking all the above estimates with the convergence to the critical manifold (Lemma \ref{Lem.Bootstrap}). 
\end{proof} 

\section{Numerical implementation}\label{sec5}
We aim to implement some numerical examples to illustrate the proposed analysis, where we focus on the convergence rate of the fast reaction limit passing from the system \eqref{System.OrginalSys}-\eqref{Condition.Initial} to  \eqref{System.LimitingSys}-\eqref{Condition.LimInitialCond} as given in Theorem \ref{Theo.ConRate}. The parabolic systems are solved by using the MATLAB\textsuperscript{\textregistered} PDE solver ``pdepe".  
To facilitate calculations, we consider our problems in the one-dimensional domain $\Omega=(0,L)$ and the following initial data 
\begin{align*}
    (R_{10},R_{20},S_0)=(1+\sin(10x), 1+\cos(20x), 1+\cos(x)),
\end{align*}
for $x\in (0,L)$. 
The systems \eqref{System.OrginalSys}-\eqref{Condition.Initial} and  \eqref{System.LimitingSys}-\eqref{Condition.LimInitialCond} are now considered in one dimension and numerically solved  corresponding to the fixed parameters given in Table  
\ref{tab:parameters}.

\begin{table}[h]
    \centering
    \begin{tabular}{|c|c|c|} 
    \hline
        $d_{R_1}$ = 0.0100 &  $\gamma_i$ = 1.0 &  $\mu$ = 5.0    \\
        \hline
         $d_{R_2}$ = 0.0021  & $\eta_i$ = 1.0  & $\widehat{R}$ = 1.0      \\
         \hline
        $d_{S}$ = 0.0040 &  $L$ = 1.0 &  $\rho$ = 0.2           \\
        \hline
    \end{tabular}
    \caption{Parameters employed in numerical implementation ($i=1,2$). }
    \label{tab:parameters}
\end{table}

% \begin{figure}[H]
%     \centering
%     \subfigure{ 
% \includegraphics[width=0.9\textwidth]{fig/su/10e-0.png} % 45% of the text width
%         \label{fig:Dif_subfig1}}
%     \subfigure{ 
% \includegraphics[width=0.9\textwidth]{fig/su/10e-3.png} % 45% of the text width
%         \label{fig:Dif_subfig2}}
%     \caption{The differences  $R^\varepsilon-R$  and $S^\varepsilon-S$, respectively on the left  and right hands.}
%     \label{fig:Dif_evolution}
% \end{figure}

 Figure \ref{fig:evolution} represents the evolution of the spatial norms $\|R^\varepsilon(t)-R(t)\|_{L^2(\Omega)}$ and $\|S^\varepsilon(t)-S(t)\|_{L^2(\Omega)}$ as $\varepsilon=10^{-k}$, $-3\le k\le 0$. After a short time, these norms rapidly become bigger for $k=0,-1$, but keep small for $k=-2,-3$. The first norm gets small quickly, while the second one needs a much longer time to get smaller, i.e. the density of the toxicity $S^\varepsilon$ needs a much longer time converging to its limit as $\varepsilon \to 0$ and mainly contributes to the rate of the fast reaction limit passing from \eqref{System.OrginalSys}-\eqref{Condition.Initial} to  \eqref{System.LimitingSys}-\eqref{Condition.LimInitialCond}.

 \begin{figure}[H]
    \centering
    \subfigure[$\varepsilon=10^0$] 
{\includegraphics[width=0.45\textwidth]{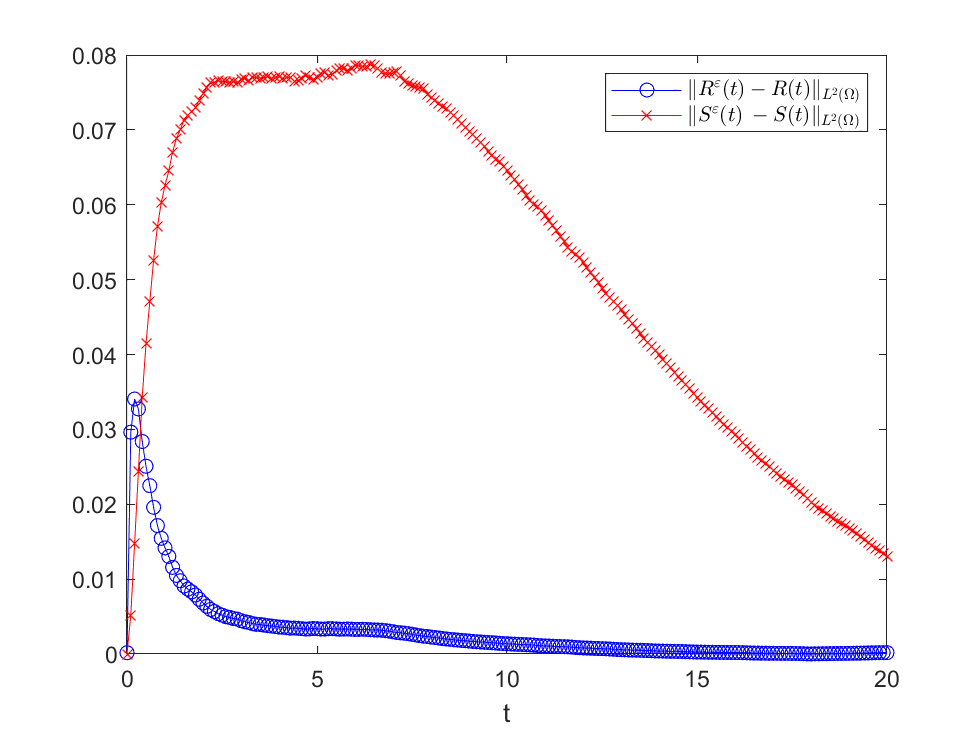} % 45% of the text width
        \label{fig:subfig1}
    }
    \hfill
    \subfigure[$\varepsilon=10^{-1}$]{
        \includegraphics[width=0.45\textwidth]{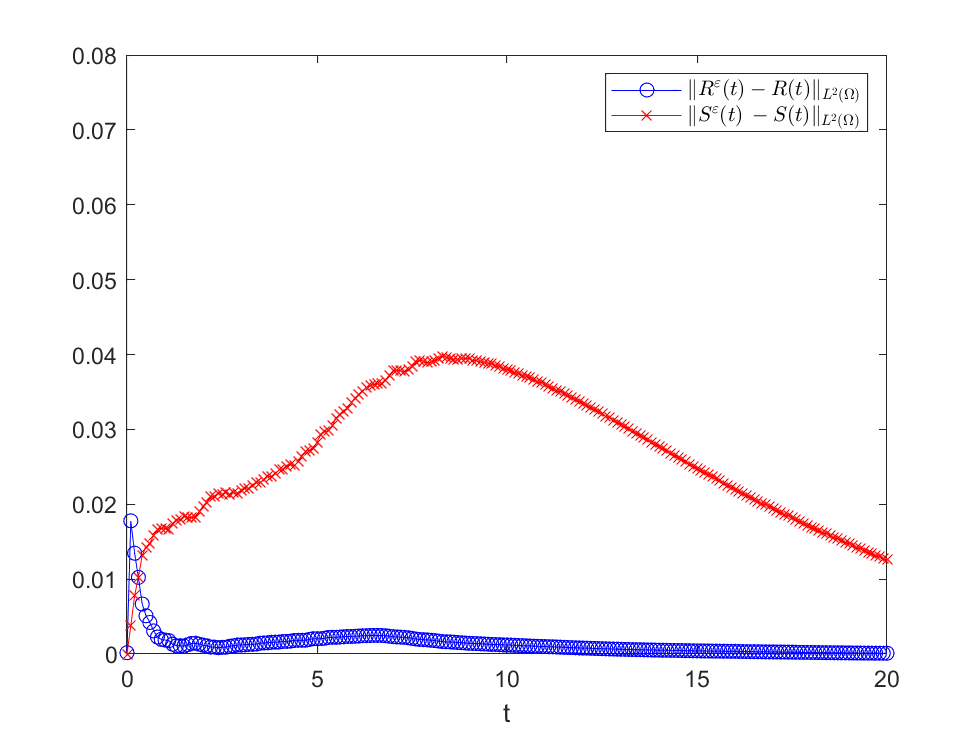} % 45% of the text width
        \label{fig:subfig2}
    }
%%%%%%%%%%%%%%%%%
    \subfigure[$\varepsilon=10^{-2}$] 
{\includegraphics[width=0.45\textwidth]{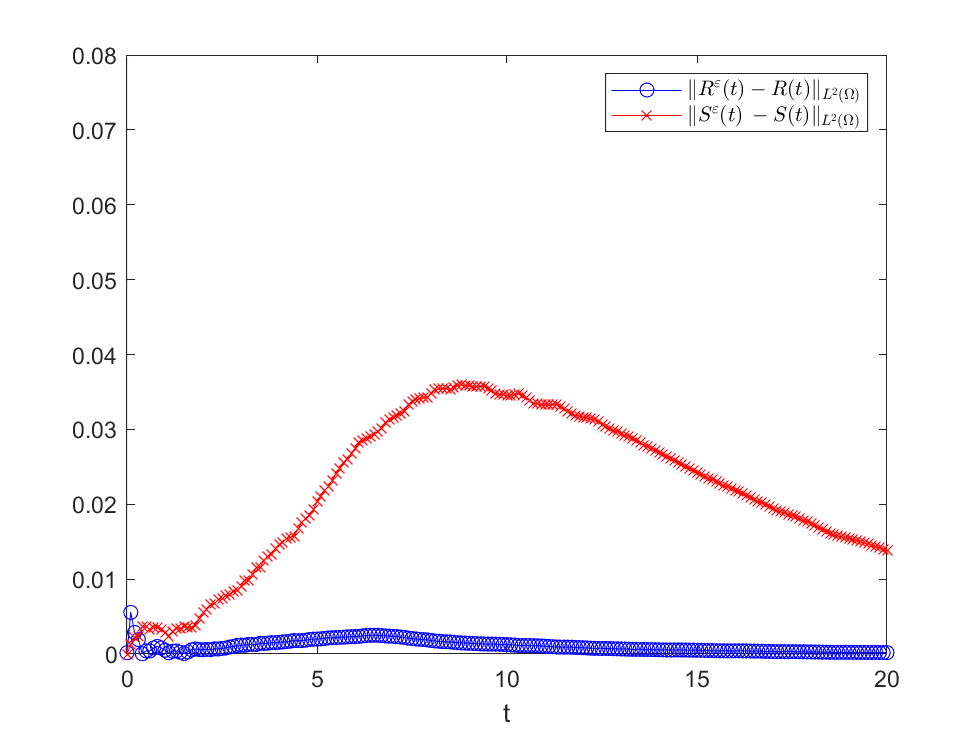} % 45% of the text width
        \label{fig:subfig1}
    }
    \hfill
    \subfigure[$\varepsilon=10^{-3}$]{
\includegraphics[width=0.45\textwidth]{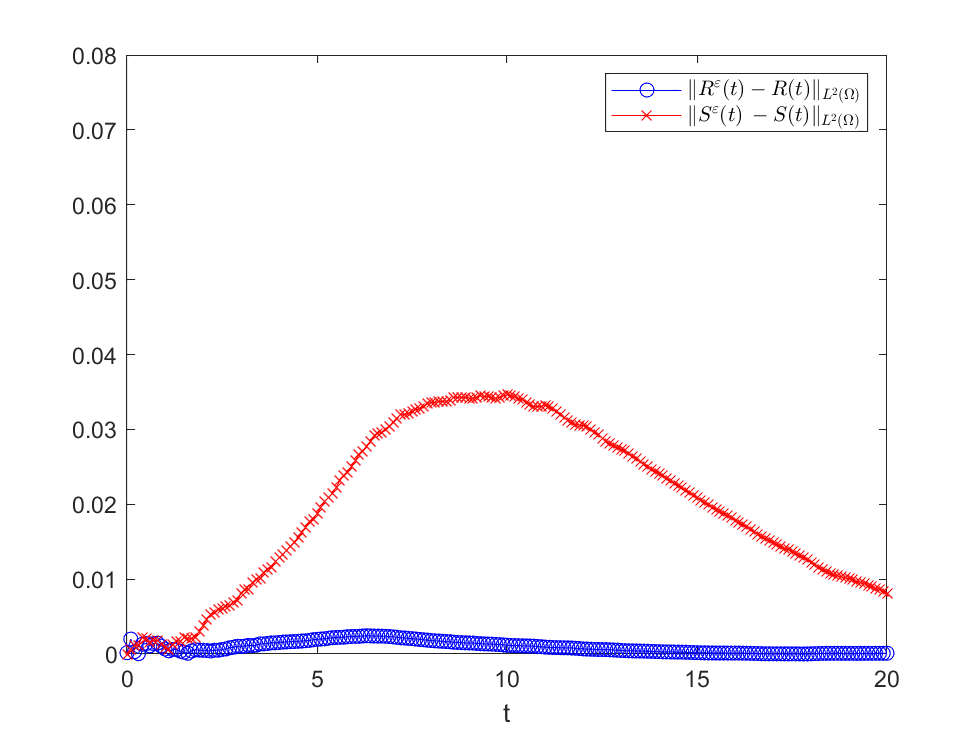} % 45% of the text width
        \label{fig:subfig2}
    }
    \caption{The evolution of $\|R^\varepsilon(t)-R(t)\|_{L^2(\Omega)}$ and $\|S^\varepsilon(t)-S(t)\|_{L^2(\Omega)}$.}
    \label{fig:evolution}
\end{figure}

\vspace*{-0.3cm}

\begin{figure}[H]
    \centering
    \includegraphics[width=0.65\textwidth]{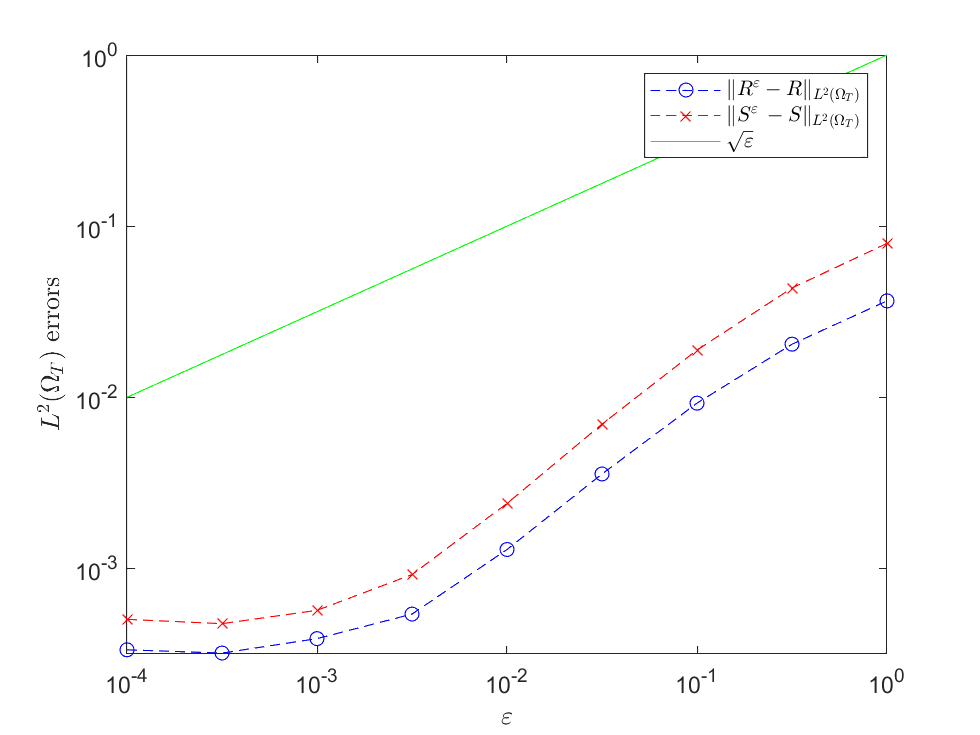} % Adjust the width as needed
    \caption{The convergence rates in $L^2(\Omega_T)$ of the fast reaction limit passing from  \eqref{System.OrginalSys}-\eqref{Condition.Initial} to  \eqref{System.LimitingSys}-\eqref{Condition.LimInitialCond} for $\varepsilon=10^{-k/2}$, $k=0,1,\dots,8$.}
    \label{fig:rate}
\end{figure}

Figure \ref{fig:rate} represents the convergence rate of the fast reaction limit considered in Theorem \ref{Theo.ConRate}, i.e. the rates $\log(\|R^\varepsilon-R\|_{L^2(\Omega_T)})/\log 10$ and $\log(\|S^\varepsilon-S\|_{L^2(\Omega_T)})/\log 10$. 
As expected, the $L^2(\Omega_T)$-errors behave with the order of magnitude of $\sqrt{\varepsilon}$, which fits the results given in Theorem \ref{Theo.ConRate}. The numerical solver works until $\varepsilon =  10^{-7/2}$. For smaller values of $\varepsilon$, the figure is not meaningful since the  $L^2(\Omega_T)$-errors do not behave as $\sqrt{\varepsilon}$ any more, which may be caused by the numerical errors and the highly nonlinear structure of the system. 

\medskip

\noindent{\textbf{Acknowledgement.}} 
The third and last authors are funded by the FWF project ``Quasi-steady-state approximation for PDE", number I-5213.

\newcommand{\etalchar}[1]{$^{#1}$}
 \newcommand{\noop}[1]{}

\end{document}